\documentclass[11pt]{article}
\usepackage[a4paper]{geometry}
\usepackage{amssymb,latexsym,amsmath,amsfonts,amsthm}
\usepackage{graphicx,overpic}
\usepackage{dsfont}
\usepackage{mathrsfs}
\usepackage{enumerate}
\usepackage{epsfig}
\usepackage{epstopdf}
\usepackage{booktabs}
\usepackage{mathtools}
\usepackage{multirow}
\usepackage{pifont}
\usepackage{comment,verbatim}
\usepackage{hyperref}
\usepackage{bbm}
\usepackage[toc,page]{appendix}

\newtheorem{theorem}{Theorem}[section]
\newtheorem{lemma}[theorem]{Lemma}

\theoremstyle{definition}

\newtheorem{example}[theorem]{Example}

\theoremstyle{remark}

\newtheorem{remark}[theorem]{Remark}

\numberwithin{equation}{section}

\usepackage{color}

\usepackage[normalem]{ulem}

\begin{document}
\title{Hermite spectral approximation for functions with endpoint singularities using exponential transforms}
\author{Haiyong Wang\footnotemark[1]~\footnotemark[2]}
\date{}
\maketitle

\footnotetext[1]{School of Mathematics and Statistics, Huazhong
University of Science and Technology, Wuhan 430074, P. R. China.
E-mail: \texttt{haiyongwang@hust.edu.cn}}

\footnotetext[2]{Hubei Key Laboratory of Engineering Modeling and
Scientific Computing, Huazhong University of Science and Technology,
Wuhan 430074, P. R. China}

\begin{abstract}
In this paper we introduce Hermite spectral approximation for functions with endpoint singularities using exponential transforms, including single exponential (SE), double exponential (DE) and error function (EF) transforms, and present a comprehensive convergence analysis for these approximations without and with scaling. In the case without scaling, we show that these methods converge at some root-exponential rate. In the case with scaling, we derive optimal scaling factors for each of exponential transforms and show that the convergence rate of Hermite spectral approximation can be significantly improved. Numerical comparisons with sinc method are present and it is shown that Hermite method has comparable or superior accuracy performance when using the same number of terms. Extensions to quadrature and rootfinding algorithm are also discussed.
\end{abstract}

{\bf Keywords:} Hermite approximation, single exponential, double exponential, endpoint singularities, sinc method

\vspace{0.05in}

{\bf AMS subject classifications:} 41A05, 41A25, 65E10

\section{Introduction}\label{sec:introduction}
Functions with singularities arise in many problems such as partial differential equations (PDEs) in regions with corners, Volterra and Fredholm integral equations with singular kernels, and the Dirichlet problem for fractional Laplacian, to name a few. It is known that numerical methods based on polynomials, such as finite element and spectral methods, exhibit slow convergence behavior for such functions. For example, for $f(x)=x^{\alpha}$, where $x\in[0,1]$ and $\alpha>0$ is not an integer, its best and Chebyshev and Legendre polynomial approximations converge only at the rate $O(n^{-2\alpha})$ as $n\rightarrow\infty$ in the maximum norm \cite{Wang2021,Wang2023a}, where $n$ is the degree of these approximations, and therefore more accurate and efficient methods are highly desirable.

Resolution of singularities has received considerable attention in the past few decades and some efficient methods, such as rational approximation \cite{Gopal2019,Herremans2023,Newman1964,Stahl2003}, M\"{u}ntz approximation \cite{Borwein1994,Cui2025} and variable transform methods \cite{Adcock2014,Chen2022,Chen2023,Okayama2013,Stenger1993,Tanaka2009}, have been developed. More specifically, Newman proved in \cite{Newman1964} that rational approximation of type $(n,n)$, i.e., $r(x)=p(x)/q(x)$ with $p(x)$ and $q(x)$ being polynomials of degree $n$, for the function $f(x)=|x|$ on the interval $[-1,1]$ converges at least at the rate $O(\exp(-\sqrt{n}))$. Building upon Newman's result, Stahl in \cite{Stahl2003} proved that the optimal convergence rate of rational approximation of type $(n,n)$ for $f(x)=|x|^{\alpha}$, where $\alpha>0$ is not an even integer and $x\in[-1,1]$, is $O(\exp(-\pi\sqrt{\alpha n}))$. Note that rational approximation of type $(2n,2n)$ to $f(x)=|x|^{2\alpha}$ on $[-1,1]$ is equivalent to rational approximation of type $(n,n)$ to $f(x)=x^{\alpha}$ on $[0,1]$, and thus the optimal convergence rate of rational approximation of type $(n,n)$ to $f(x)=x^{\alpha}$ on $[0,1]$ is $O(\exp(-2\pi\sqrt{\alpha n}))$, which is clearly much faster than its polynomial counterpart. More recently, a rational approximation called ``lighting plus polynomial approximation'' was proposed in \cite{Herremans2023} for $f(x)=x^{\alpha}$ on $[0,1]$. By using preassigned poles which exponentially clustered near the singularities, this method can be efficiently implemented by solving a least-squares problem. In particular, when the poles are tapered exponentially clustered near the singularities, this method achieves the convergence rate of the best rational approximation.

On the other hand, variable transform method has also received much attention. The basic idea is to apply some bijective mappings to transform the underlying function to some new functions on unbounded domains with fast decay behaviors at infinity, and then apply some accurate approximations to the transplanted functions over the unbounded domains. Variable transform method includes the single exponential (SE) sinc method \cite{Stenger1993}, double exponential (DE) sinc method \cite{Okayama2013,Tanaka2009}, Chebyshev approximation on a truncated interval \cite{Adcock2014}, log orthogonal approximations \cite{Chen2022,Chen2023}. All these methods achieve root-exponential or faster convergence for functions with endpoint singularities.

In this work, we introduce Hermite spectral approximation for functions with endpoint singularities using exponential transforms and provide a comprehensive convergence analysis of such approximations without and with scaling. Three exponential transforms, including single-exponential (SE), double-exponential (DE) and error function (EF) transforms which are widely used in numerical integration and sinc method (see, e.g., \cite{Okayama2013,Stenger1993,Takahasi1973,Takahasi1974,Tanaka2009}), are considered. In the case without scaling, we show that these Hermite spectral approximations converge at some root-exponential rate, i.e., $O(\exp(-\nu\sqrt{n}))$ for some $\nu>0$ and $n$ is the degree of the approximation. In the case with scaling, we derive optimal scaling factor for SE transform based on the Hermite approximation result in \cite{Hille1940} and for DE and EF transforms based on the recent scaling optimized Hermite approximation result in \cite{Hu2026}. In particular, we show that the convergence rate of Hermite spectral approximation using DE transform can be improved to almost-exponential, i.e., $O(\exp(-\nu n/\log n))$ for some $\nu>0$. Numerical comparison of Hermite approximations using SE and DE transforms with their sinc counterparts is present, which shows that the former has comparable or superior accuracy performance than the latter when using the same number of terms.

The rest of this paper is organized as follows. In section \ref{sec:Herm}, we introduce Hermite spectral approximation on bounded domains using transforms and present a general convergence result. Three exponential transforms are considered and more precise convergence rates are present in section \ref{sec:ExpTrans}. We analyze Hermite spectral approximation with scaling in section \ref{sec:ScaledHerm} and develop their interpolation analog in section \ref{sec:Interp}. We present two extensions in section \ref{sec:Exten} and finish this work with some concluding remarks in section \ref{sec:conclusion}.

\section{Hermite spectral approximation for functions with endpoint singularities}\label{sec:Herm}
The Hermite functions $\psi_n(x)$ are defined by
\begin{equation}\label{def:HermFun}
\psi_n(x) = \frac{(-1)^n}{\sqrt{\gamma_{n}}} e^{x^2/2} \frac{d^n}{d x^n} e^{-x^2},
\end{equation}
where $n\in\mathbb{N}_0:=\{0,1,2,\ldots\}$ and $\gamma_{n}=2^n n!\sqrt{\pi}$. It is well known that $\{\psi_n\}_{n=0}^{\infty}$ 
form a complete orthonormal basis of the Hilbert space $L^2(\mathbb{R})$, i.e.,
\begin{align}
\int_{\mathbb{R}} \psi_j(x) \psi_k(x) dx = \delta_{j,k},
\end{align}
where $\delta_{j,k}$ is the Kronecker delta.

Let $I:=(a,b)$ be a finite interval, i.e., $-\infty<a<b<\infty$ and let $\phi$ be a bijective mapping from $(-\infty,\infty)$ to $(a,b)$ and satisfies $\phi(-\infty)=a$ and $\phi(\infty)=b$. Moreover, let $\varphi$ denote the inverse of $\phi$, i.e., $\varphi=\phi^{-1}$. For $n\in\mathbb{N}_0$, we introduce the following system
\begin{equation}\label{def:PsiTran}
\Psi_n(x) = (\psi_n\circ\varphi)(x), \quad x\in I.
\end{equation}
It is easy to verify that they satisfy the following properties:
\begin{itemize}
\item Orthonormality
\begin{equation}
\int_{a}^{b} \Psi_n(x) \Psi_m(x) \varphi'(x) dx = \delta_{n,m}. \nonumber
\end{equation}

\item Recurrence
\begin{equation}
\varphi(x) \Psi_{n}(x) = \sqrt{\frac{n}{2}} \Psi_{n-1}(x) + \sqrt{\frac{n+1}{2}} \Psi_{n+1}(x). \nonumber
\end{equation}

\item Derivative
\begin{equation}
\Psi'_{n}(x) = \varphi'(x) \left( \sqrt{\frac{n}{2}} \Psi_{n-1}(x) - \sqrt{\frac{n+1}{2}} \Psi_{n+1}(x) \right). \nonumber
\end{equation}

\item Bound
\begin{equation}
|\Psi_{n}(x)| \leq \pi^{-1/4}, \quad  x \in I. \nonumber
\end{equation}
\end{itemize}

Now we turn to the problem of approximating functions with endpoint singularities. Let $f$ be analytic on $I$ and continuous on $\overline{I}$ and let $\mathbb{M}_n=\mathrm{span}\{\Psi_k\}_{k=0}^{n}$. The orthogonal projection of $f$ onto the space $\mathbb{M}_n$ is given by
\begin{equation}\label{def:MapProj}
(\Pi_nf)(x) = \sum_{k=0}^{n} a_k \Psi_k(x), \quad  a_k = \int_{I} f(x) \Psi_k(x) \varphi'(x) dx.
\end{equation}
Let $\mathcal{S}_{\rho}$ denote the strip region of the form
\begin{align}
\mathcal{S}_{\rho} = \{z\in \mathbb{C}: ~ |\Im(z)|< \rho \}.
\end{align}
Below we establish a general convergence result for the approximation \eqref{def:MapProj}. Throughout the paper, we use the notation $\mathcal{K}$ to denote some generic positive constant which is independent of $n$ and may take different values in different places.
\begin{theorem}\label{thm:MapHermConv}
If $f$ satisfies the following conditions:
\begin{itemize}
\item[\rm(i)] $f$ is analytic on $\phi(\mathcal{S}_{\rho})$ for some $\rho>0$;

\item[\rm(ii)] $|f(x)|=O(\exp(-\tau|\varphi(x)|^{\kappa}))$ for some $\tau>0$ and $\kappa\geq1$ as $x\rightarrow{a}^{+}$ and $x\rightarrow{b}^{-}$.
\end{itemize}
Then, it holds
\begin{align}\label{eq:MapHermConv}
\|f - \Pi_nf \|_{L^{\infty}(I)} \leq \mathcal{K} \sqrt{n} \exp(-\nu\sqrt{n}),
\end{align}
where $\nu=\sqrt{2}\rho$ whenever $\kappa>1$ and $\nu=\sqrt{2}\min\{\tau,\rho\}$ whenever $\kappa=1$.
\end{theorem}
\begin{proof}
Let $s=\varphi(x)$, where $s\in(-\infty,\infty)$, and let $F(s)=(f\circ\phi)(s)$. We have
\[
a_k = \int_{I} f(x) \Psi_k(x) \varphi'(x) dx = \int_{\mathbb{R}} F(s) \psi_k(s) ds,
\]
and therefore, the coefficient $a_k$ is exactly the $k$-th expansion coefficient of $F(s)$ in terms of $\{\psi_k(s)\}$. Moreover, by items (i) and (ii) we see that $F(s)$ is analytic in the strip $\mathcal{S}_{\rho}$ and $|F(s)|=O(\exp(-\tau|s|^\kappa))$ as $|s|\rightarrow\infty$ along the real axis. Hence, from \cite{Boyd1984,Hille1940} we deduce that $|a_k|=O(\exp(-\nu\sqrt{k}))$ and $\nu=\sqrt{2}\min\{\tau,\rho\}$ whenever $\kappa=1$ and $\nu=\sqrt{2}\rho$ whenever $\kappa>1$. Further, by the inequality $|\Psi_{n}(x)| \leq \pi^{-1/4}$ we have
\begin{align}
\|f - \Pi_nf\|_{L^{\infty}(I)} &\leq \frac{1}{\pi^{1/4}} \sum_{k=n+1}^{\infty} |a_k| \leq \mathcal{K} \sum_{k=n+1}^{\infty} e^{-\nu\sqrt{k}} \nonumber \\
&\leq \mathcal{K} \int_{n}^{\infty} e^{-\nu\sqrt{x}} dx
\leq \mathcal{K} \sqrt{n} e^{-\nu\sqrt{n}}. \nonumber
\end{align}
This ends the proof.
\end{proof}

\begin{remark}\label{rk:AlgFactor}
The algebraic factor $\sqrt{n}$ in \eqref{eq:MapHermConv} might be reduced by some more precise estimate of the coefficients $\{a_k\}$. Indeed, it was shown in \cite[Theorem~3.8]{Wang2026} that $|a_k|=O(k^{-1/4}e^{-\nu\sqrt{k}})$ with $\nu=\sqrt{2}\rho$ if $F(s)$ is analytic in the strip $\mathcal{S}_{\rho}$ and $|e^{z^2/2}F(z)|\leq \mathcal{K}|z|^{\sigma}$ for some $\sigma\in\mathbb{R}$ as $|z|\rightarrow\infty$ within the strip and $\int_{\partial\mathcal{S}_{\rho}}|e^{z^2/2}F(z) dz|<\infty$. In this case, the error bound in \eqref{eq:MapHermConv} can be improved to
\begin{align}
\|f - \Pi_nf \|_{L^{\infty}(I)} \leq \mathcal{K} \sum_{k=n+1}^{\infty} k^{-1/4} e^{-\nu\sqrt{k}} \leq \mathcal{K} n^{1/4} e^{-\nu\sqrt{n}}, \nonumber
\end{align}
where $\nu=\sqrt{2}\rho$.
\end{remark}

\begin{remark}\label{rk:EndNonZero}
The condition (ii) in Theorem \ref{thm:MapHermConv} implies that $f(a)=f(b)=0$. If either $f(a)\neq0$ or $f(b)\neq0$, then we can define an approximation of the form
\begin{align}
(\widetilde{\Pi}_n f)(x) := (\Pi_n h)(x) + p(x),
\end{align}
where $h=f-p$ and $p$ is a polynomial of degree one satisfying $p(a)=f(a)$ and $p(b)=f(b)$. If $h$ satisfies the conditions (i) and (ii) of Theorem \ref{thm:MapHermConv}, then we have
\begin{align}\label{eq:MapHermError2}
\|f - \widetilde{\Pi}_n f\|_{L^{\infty}(I)} &= \|h - \Pi_n h\|_{L^{\infty}(I)} \leq \mathcal{K} \sqrt{n} \exp(-\nu\sqrt{n}),
\end{align}
where $\nu$ is defined in \eqref{eq:MapHermConv}, and thus root-exponential convergence can still be achieved.
\end{remark}

\section{Three exponential transforms}\label{sec:ExpTrans}
In this section we consider three exponential transforms, including single exponential (SE), double exponential (DE) and error function (EF) transforms. Precise convergence rates of Hermite approximations based on these transforms will be given.

\subsection{Single exponential transform}
We first consider the single exponential (SE) transform and its inverse
\begin{align}\label{def:SE}
\phi_{\mathrm{SE}}(s) &= \frac{b-a}{2}\tanh\left(\frac{s}{2}\right) + \frac{b+a}{2},
\quad  \varphi_{\mathrm{SE}}(x) = \log\left(\frac{x-a}{b-x}\right),
\end{align}
and introduce the following SE-Hermite functions
\begin{equation}\label{def:SEHF}
\Psi_n^{\mathrm{SE}}(x) := (\psi_k\circ\varphi_{\mathrm{SE}})(x).
\end{equation}
The orthogonal projection in \eqref{def:MapProj} based on SE transform is given by
\begin{align}\label{def:SEProj}
(\Pi_n^{\mathrm{SE}}f)(x) &= \sum_{k=0}^{n} a_k^{\mathrm{SE}} \Psi_n^{\mathrm{SE}}(x), \quad  a_k^{\mathrm{SE}} = \int_{I} f(x) \Psi_n^{\mathrm{SE}}(x) \varphi'_{\mathrm{SE}}(x) dx.
\end{align}
The SE-Hermite system $\{\Psi_n^{\mathrm{SE}}\}$ in \eqref{def:SEHF} has been proposed in \cite{Boyd1986} for approximating functions with endpoint singularities and was recently used to develop spectral method in \cite{WangM2025} for solving singular integral equations. However, as far as we know, a rigorous analysis on the root-exponential convergence of the approximation \eqref{def:SEProj} is still lacking. Below we state a precise convergence result for \eqref{def:SEProj}.
\begin{theorem}\label{thm:SEConv}
If $f$ is analytic on $\phi_{\mathrm{SE}}(\mathcal{S}_{\rho})$ for some $\rho\in(0,\pi)$ and $|f(x)| \leq \mathcal{K} |(x-a)^{\alpha}(b-x)^{\beta}|$ for some $\alpha,\beta>0$ and $x\in I$. Then, it holds
\begin{align}\label{eq:SEConv}
\|f - \Pi_n^{\mathrm{SE}}f \|_{L^{\infty}(I)} \leq \mathcal{K} \sqrt{n} \exp(-\nu\sqrt{n}),
\end{align}
where $\nu=\sqrt{2}\min\{\tau,\rho\}$ and $\tau=\min\{\alpha,\beta\}$.
\end{theorem}
\begin{proof}
By the condition $|f(x)| \leq \mathcal{K} |(x-a)^{\alpha}(b-x)^{\beta}|$ for $x\in I$, we can deduce that $|f(x)|\leq \mathcal{K} \exp(-\alpha|\varphi_{\mathrm{SE}}(x)|)$ as $x\rightarrow{a}^{+}$ and $|f(x)|\leq \mathcal{K} \exp(-\beta|\varphi_{\mathrm{SE}}(x)|)$ as $x\rightarrow{b}^{-}$. The desired convergence result \eqref{eq:SEConv} then follows immediately from Theorem \ref{thm:MapHermConv} with $\kappa=1$ and $\tau=\min\{\alpha,\beta\}$. This ends the proof.
\end{proof}

Clearly, the condition in Theorem \ref{thm:SEConv} implies $f(a)=f(b)=0$. If either $f(a)$ or $f(b)$ are nonzero, then by Remark \ref{rk:EndNonZero} we define the following approximation
\begin{equation}
(\widetilde{\Pi}_n^{\mathrm{SE}}f)(x) = (\Pi_n^{\mathrm{SE}} (f-p))(x) + p(x),
\end{equation}
where $p$ is a polynomial of degree one satisfying $p(a)=f(a)$ and $p(b)=f(b)$. The convergence result is stated in the following theorem.
\begin{theorem}\label{thm:SEConv2}
If $f$ is analytic on $\phi_{\mathrm{SE}}(\mathcal{S}_{\rho})$ for some $\rho\in(0,\pi)$ and
$|f(x)-f(a)|=O((x-a)^{\alpha})$ as $x\rightarrow a^{+}$ and $|f(x)-f(b)|=O((b-x)^{\beta})$ as $x\rightarrow b^{-}$ for some $\alpha,\beta\in(0,1]$ and $x\in I$. Then, it holds
\begin{align}\label{eq:SEConv2}
\|f - \widetilde{\Pi}_n^{\mathrm{SE}}f \|_{L^{\infty}(I)} \leq \mathcal{K} \sqrt{n} \exp(-\nu\sqrt{n}),
\end{align}
where $\nu=\sqrt{2}\min\{\tau,\rho\}$ and $\tau=\min\{\alpha,\beta\}$.
\end{theorem}
\begin{proof}
Clearly, $f-p$ is analytic on $\phi_{\mathrm{SE}}(\mathcal{S}_{\rho})$. Moreover, as $x\rightarrow a^{+}$, we have
\[
|f(x) - p(x)| \leq |f(x) - f(a)| + |f(a) - p(x)| = O((x-a)^{\alpha}),
\]
and, as $x\rightarrow b^{-}$,
\[
|f(x) - p(x)| \leq |f(x) - f(b)| + |f(b) - p(x)| = O((b-x)^{\beta}).
\]
and thus $f-p$ satisfies the conditions of Theorem \ref{thm:SEConv}. Note that
$$\|f - \widetilde{\Pi}_n^{\mathrm{SE}}f \|_{L^{\infty}(I)} = \|(f-p) - \Pi_n^{\mathrm{SE}} (f-p)\|_{L^{\infty}(I)},$$
the desired result \eqref{eq:SEConv2} follows immediately. This ends the proof.
\end{proof}

\subsection{Double exponential transform}
We consider the double exponential (DE) transform and its inverse
\begin{equation}\label{def:DE}
\phi_{\mathrm{DE}}(s) = \frac{b-a}{2}\tanh\left(\frac{\pi}{2}\sinh(s)\right) + \frac{b+a}{2}, \quad
\varphi_{\mathrm{DE}}(x) = \mathrm{arcsinh}\left( \frac{1}{\pi}\log\left(\frac{x-a}{b-x} \right)  \right),
\end{equation}
and introduce the following DE-Hermite functions
\begin{equation}\label{def:DEHF}
\Psi_n^{\mathrm{DE}}(x) := (\psi_k\circ\varphi_{\mathrm{DE}})(x).
\end{equation}
The orthogonal projection in \eqref{def:MapProj} based on DE transform is given by
\begin{align}\label{def:DEProj}
(\Pi_n^{\mathrm{DE}}f)(x) &= \sum_{k=0}^{n} a_k^{\mathrm{DE}} \Psi_n^{\mathrm{DE}}(x), \quad  a_k^{\mathrm{DE}} = \int_{I} f(x) \Psi_n^{\mathrm{DE}}(x) \varphi'_{\mathrm{DE}}(x) dx.
\end{align}
The convergence result is stated in the following theorem.
\begin{theorem}\label{thm:DEConv}
If $f$ is analytic on $\phi_{\mathrm{DE}}(\mathcal{S}_{\rho})$ for some $\rho\in(0,\pi/2)$ and $|f(x)| \leq \mathcal{K} |(x-a)^{\alpha}(b-x)^{\beta}|$ for some $\alpha,\beta>0$ and $x\in I$. Then, it holds
\begin{align}\label{eq:DEConv}
\|f - \Pi_n^{\mathrm{DE}}f \|_{L^{\infty}(I)} \leq \mathcal{K} \sqrt{n} \exp(-\nu\sqrt{n}),
\end{align}
where $\nu=\sqrt{2}\rho$.
\end{theorem}
\begin{proof}
By the condition $|f(x)| \leq \mathcal{K} |(x-a)^{\alpha}(b-x)^{\beta}|$ for $x\in I$, it is easily verified that
\[
|f(x)| \leq \mathcal{K} \exp\left(-\frac{\alpha\pi}{2} \exp(|\varphi_{\mathrm{DE}}(x)|) \right),
\]
as $x\rightarrow a^{+}$, and
\[
|f(x)| \leq \mathcal{K} \exp\left(-\frac{\beta\pi}{2} \exp(|\varphi_{\mathrm{DE}}(x)|) \right),
\]
as $x\rightarrow b^{-}$. Therefore, $f(x)$ satisfies the item (ii) in Theorem \ref{thm:MapHermConv} with the exponent $\kappa>0$ being arbitrarily large. The desired convergence result \eqref{eq:DEConv} follows immediately from Theorem \ref{thm:MapHermConv}. This ends the proof.
\end{proof}

If either $f(a)$ or $f(b)$ are nonzero, then we define the following approximation
\begin{equation}
(\widetilde{\Pi}_n^{\mathrm{DE}}f)(x) = (\Pi_n^{\mathrm{DE}}(f-p))(x) + p(x),
\end{equation}
where $p$ is a polynomial of degree one satisfying $p(a)=f(a)$ and $p(b)=f(b)$. The convergence result is stated in the following theorem.
\begin{theorem}\label{thm:DEConv2}
If $f$ is analytic on $\phi_{\mathrm{DE}}(\mathcal{S}_{\rho})$ for some $\rho\in(0,\pi/2)$ and
$|f(x)-f(a)|=O((x-a)^{\alpha})$ as $x\rightarrow a^{+}$ and $|f(x)-f(b)|=O((b-x)^{\beta})$ as $x\rightarrow b^{-}$ for some $\alpha,\beta>0$ and $x\in I$. Then, it holds
\begin{align}\label{eq:DEConv2}
\|f - \widetilde{\Pi}_n^{\mathrm{DE}}f \|_{L^{\infty}(I)} \leq \mathcal{K} \sqrt{n} \exp(-\nu\sqrt{n}),
\end{align}
where $\nu=\sqrt{2}\rho$.
\end{theorem}
\begin{proof}
It is similar to the proof of Theorem \ref{thm:SEConv2}. We omit the details.
\end{proof}

\subsection{Error function transform}
We consider the error function (EF) transform and its inverse
\begin{equation}\label{def:EF}
\phi_{\mathrm{EF}}(s) = \frac{b-a}{2}\mathrm{erf}(s) + \frac{b+a}{2}, \quad  \varphi_{\mathrm{EF}}(x) = \mathrm{erf}^{-1}\left( \frac{2x-a-b}{b-a} \right),
\end{equation}
where $\mathrm{erf}(z)$ is the error function \cite[Chapter~7]{Olver2010}. We introduce the following EF-Hermite functions
\begin{equation}\label{def:EFHF}
\Psi_n^{\mathrm{EF}}(x) := (\psi_k\circ\varphi_{\mathrm{EF}})(x) .
\end{equation}
The orthogonal projection in \eqref{def:MapProj} based on EF transform is given by
\begin{align}\label{def:EFProj}
(\Pi_n^{\mathrm{EF}}f)(x) &= \sum_{k=0}^{n} a_k^{\mathrm{EF}} \Psi_n^{\mathrm{EF}}(x), \quad  a_k^{\mathrm{EF}} = \int_{I} f(x) \Psi_n^{\mathrm{EF}}(x) \varphi_{\mathrm{EF}}'(x) dx.
\end{align}
In the sequel $\mathrm{erfc}(z)$ always denotes the complementary error function. The convergence result of \eqref{def:EFProj} is stated in the following theorem.
\begin{theorem}\label{thm:EFConv}
If $f$ is analytic on $\phi_{\mathrm{EF}}(\mathcal{S}_{\rho})$ for some $\rho\in(0,\rho^{*})$, where $\rho^{*}$ is the imaginary part of the root of $\mathrm{erfc}(z)$ that is nearest to the real axis, and $|f(x)|\leq\mathcal{K} |(x-a)^{\alpha}(b-x)^{\beta}|$ for some $\alpha,\beta>0$ and $x\in I$. Then, it holds
\begin{align}\label{eq:EFConv}
\|f - \Pi_n^{\mathrm{EF}}f \|_{L^{\infty}(I)} \leq \mathcal{K} \sqrt{n} \exp(-\nu\sqrt{n}),
\end{align}
where $\nu=\sqrt{2}\rho$.
\end{theorem}
\begin{proof}
By the asymptotics of the error function
\begin{align}
\mathrm{erf}(x) \sim \left\{
            \begin{array}{ll}
{\displaystyle -1 - \frac{e^{-x^2}}{\sqrt{\pi}x} (1 + O(x^{-2}))},   & \hbox{$x\rightarrow-\infty$,}    \\[3ex]
{\displaystyle  1 - \frac{e^{-x^2}}{\sqrt{\pi}x} (1 + O(x^{-2}))},   & \hbox{$x\rightarrow \infty$,}
            \end{array}
            \right.  \nonumber
\end{align}
and the condition $|f(x)|\leq\mathcal{K} |(x-a)^{\alpha}(b-x)^{\beta}|$ for $x\in I$,
it is easily verified that $|f(x)| \leq \mathcal{K}\exp(-\alpha|\varphi_{\mathrm{EF}}(x)|^2)$ as $x\rightarrow a^{+}$, and $|f(x)|\leq\mathcal{K}\exp(-\beta|\varphi_{\mathrm{EF}}(x)|^2)$ as $x\rightarrow b^{-}$. Hence, the desired convergence result \eqref{eq:EFConv} follows from Theorem \ref{thm:MapHermConv} with $\kappa=2$. This ends the proof.
\end{proof}

\begin{remark}
From \cite[Table~7.13.2]{Olver2010} we know that the roots of $\mathrm{erfc}(z)$ which are nearest to the real axis are $\{z,\overline{z}\}$, where $z\approx-1.354810128 + 1.991466843\mathrm{i}$ and $\mathrm{i}$ is the imaginary unit. Thus, $\rho^{*}\approx1.991466843$ in the above Theorem.
\end{remark}

If either $f(a)$ or $f(b)$ are nonzero, then we can define the following approximation
\begin{equation}
(\widetilde{\Pi}_n^{\mathrm{EF}}f)(x) = (\Pi_n^{\mathrm{EF}} (f-p))(x) + p(x),
\end{equation}
where $p$ is a polynomial of degree one satisfying $p(a)=f(a)$ and $p(b)=f(b)$. The convergence result is stated in the following theorem.
\begin{theorem}\label{thm:EFConv2}
If $f$ is analytic on $\phi_{\mathrm{EF}}(\mathcal{S}_{\rho})$ for some $\rho\in(0,\rho^{*})$, and
$|f(x)-f(a)|=O((x-a)^{\alpha})$ as $x\rightarrow a^{+}$ and $|f(x)-f(b)|=O((b-x)^{\beta})$ as $x\rightarrow b^{-}$ for some $\alpha,\beta>0$ and $x\in I$. Then, it holds
\begin{align}\label{eq:EFConv2}
\|f - \widetilde{\Pi}_n^{\mathrm{EF}}f \|_{L^{\infty}(I)} \leq \mathcal{K} \sqrt{n} \exp(-\nu\sqrt{n}),
\end{align}
where $\nu=\sqrt{2}\rho$.
\end{theorem}
\begin{proof}
It is similar to the proof of Theorem \ref{thm:SEConv2}. We omit the details.
\end{proof}

\subsection{Examples}
We present examples to illustrate the root-exponential convergence rates of $\Pi_n^{\mathrm{SE}}f$, $\Pi_n^{\mathrm{DE}}f$ and $\Pi_n^{\mathrm{EF}}f$. We choose $I=(0,1)$ and all computations are performed using the Advanpix Multiprecision Computing Toolbox with quadruple precision.

\begin{example}
We consider $f(x)=x^{\alpha}(1-x)^{\beta}$ with $\alpha,\beta>0$. Clearly, this function satisfies the conditions of Theorems \ref{thm:SEConv}, \ref{thm:DEConv} and \ref{thm:EFConv} and corresponds to $\rho=\pi-\epsilon$, $\rho=\pi/2-\epsilon$ and $\rho=\rho^{*}-\epsilon$, respectively, with $\epsilon>0$ being arbitrarily close to zero. In Figure \ref{fig:Exam1} we plot their maximum errors as a function of $\sqrt{n}$ for $(\alpha,\beta)=(\sqrt{2}/5,\sqrt{2}/5)$ and $(\alpha,\beta)=(1/3,1/2)$. Moreover, we also plot the predicted convergence rates, i.e., $O(n^{\kappa}\exp(-\nu\sqrt{n}))$ with $\kappa=0$ and $\nu=\sqrt{2}\min\{\pi,\alpha,\beta\}$ for $\Pi_n^{\mathrm{SE}}f$, and $\kappa=0$ and $\nu=\sqrt{2}\pi/2$ for $\Pi_n^{\mathrm{DE}}f$, and $\kappa=-3/4$ and $\nu=\sqrt{2}\rho^{*}$ for $\Pi_n^{\mathrm{EF}}f$. We see that, up to some small algebraic factors, the convergence rates of $\Pi_n^{\mathrm{SE}}f$, $\Pi_n^{\mathrm{DE}}f$ and $\Pi_n^{\mathrm{EF}}$ are all consistent with the predicted rates.
\end{example}

\begin{figure}[htbp]
\centering
\includegraphics[width=0.49\textwidth,height=0.42\textwidth]{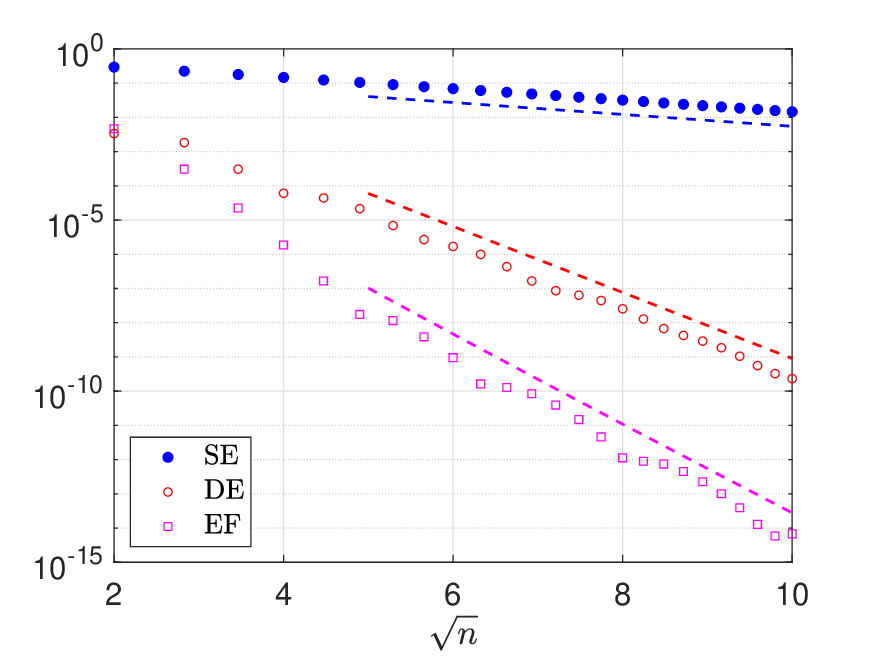}
\includegraphics[width=0.49\textwidth,height=0.42\textwidth]{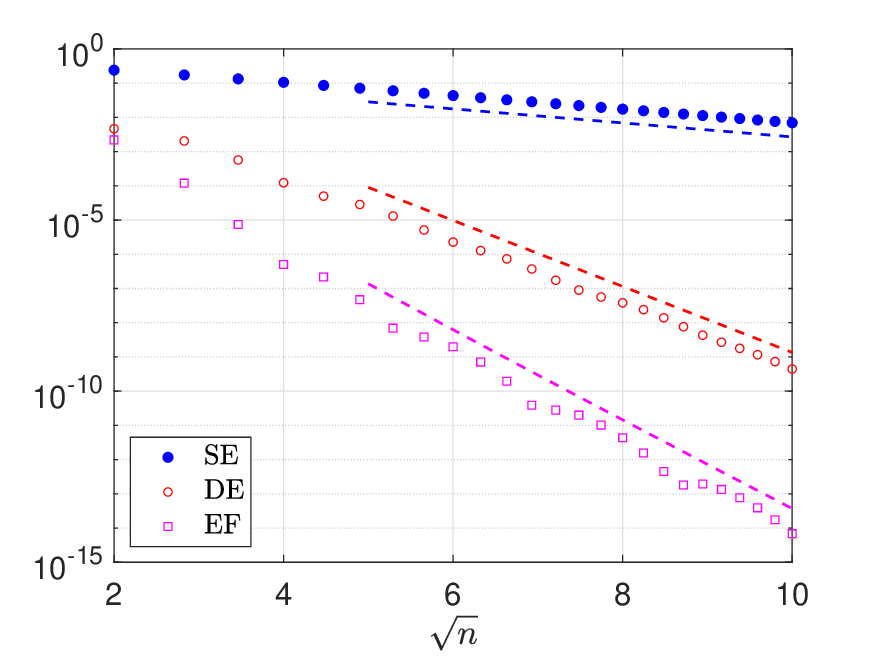}
\caption{Maximum errors of $\Pi_n^{\mathrm{SE}}f$ (dots), $\Pi_n^{\mathrm{DE}}f$ (circles) and $\Pi_n^{\mathrm{EF}}$ (boxes) as a function of $\sqrt{n}$ for $f(x) = x^{\alpha}(1-x)^{\beta}$. Here $(\alpha,\beta)=(\sqrt{2}/5,\sqrt{2}/5)$ (left) and $(\alpha,\beta)=(1/3,1/2)$ (right). The dashed lines are $O(n^{\kappa}\exp(-\nu\sqrt{n}))$ with $\kappa=0$ and $\nu=\sqrt{2}\min\{\pi,\alpha,\beta\}$ for $\Pi_n^{\mathrm{SE}}f$, and $\kappa=0$ and $\nu=\sqrt{2}\pi/2$ for $\Pi_n^{\mathrm{DE}}f$, and $\kappa=-3/4$ and $\nu=\sqrt{2}\rho^{*}$ for $\Pi_n^{\mathrm{EF}}f$. }\label{fig:Exam1}
\end{figure}

\begin{example}
We consider $f(x)=x^{\alpha}\log^{\beta}(x)$ with $\alpha>0$ and $\beta\in\mathbb{N}$. This function, which has an algebraic-logarithmic singularity at $x=0$, also satisfies the conditions of Theorems \ref{thm:SEConv}, \ref{thm:DEConv} and \ref{thm:EFConv} and corresponds to $\rho=\pi-\epsilon$, $\rho=\pi/2-\epsilon$ and $\rho=\rho^{*}-\epsilon$, respectively. In Figure \ref{fig:Exam2} we plot the maximum errors of $\Pi_n^{\mathrm{SE}}f$, $\Pi_n^{\mathrm{DE}}f$ and $\Pi_n^{\mathrm{EF}}$ as a function of $\sqrt{n}$ for $(\alpha,\beta)=(1/2,1)$ and $(\alpha,\beta)=(1,1)$. Moreover, we also plot the predicted convergence rates $O(n^{\kappa}\exp(-\nu\sqrt{n}))$ with $\kappa=1/2$ and $\nu=\sqrt{2}\min\{\pi,\alpha,\beta\}$ for $\Pi_n^{\mathrm{SE}}f$, and $\kappa=2/3$ and $\nu=\sqrt{2}\pi/2$ for $\Pi_n^{\mathrm{DE}}f$, and $\kappa=-3/4$ for $(\alpha,\beta)=(1/2,1)$ and $\kappa=-1$ for $(\alpha,\beta)=(1,1)$ and $\nu=\sqrt{2}\rho^{*}$ for $\Pi_n^{\mathrm{EF}}f$. We see again that, up to some small algebraic factors, the convergence rates of $\Pi_n^{\mathrm{SE}}f$, $\Pi_n^{\mathrm{DE}}f$ and $\Pi_n^{\mathrm{EF}}$ are all consistent with the predicted rates.
\end{example}

\begin{figure}[htbp]
\centering
\includegraphics[width=0.49\textwidth,height=0.42\textwidth]{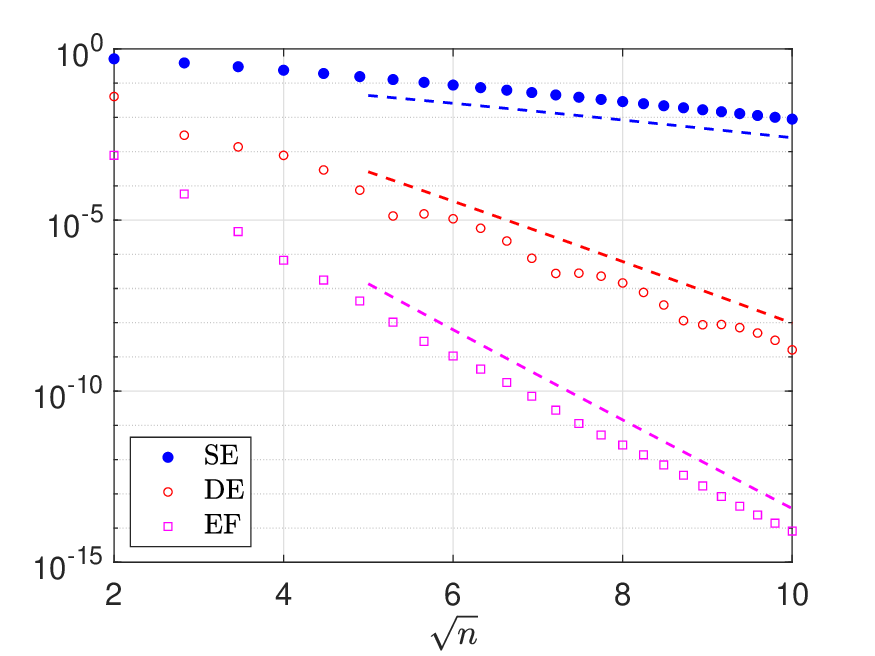}
\includegraphics[width=0.49\textwidth,height=0.42\textwidth]{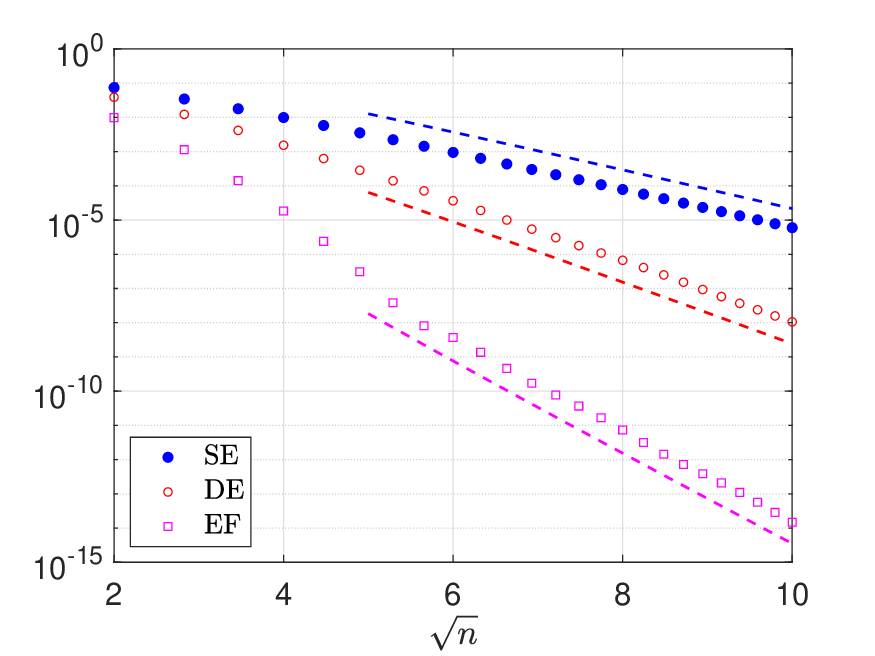}
\caption{Maximum error of $\Pi_n^{\mathrm{SE}}f$ (dots), $\Pi_n^{\mathrm{DE}}f$ (circles) and $\Pi_n^{\mathrm{EF}}$ (boxes) as a function of $\sqrt{n}$ for $f(x) = x^{\alpha}\log^{\beta}(x)$. Here $(\alpha,\beta)=(1/2,1)$ (left) and $(\alpha,\beta)=(1,1)$ (right). The dashed lines are $O(n^{\kappa}\exp(-\nu\sqrt{n}))$ with $\kappa=1/2$ and $\nu=\sqrt{2}\min\{\pi,\alpha,\beta\}$ for $\Pi_n^{\mathrm{SE}}f$, and $\kappa=2/3$ and $\nu=\sqrt{2}\pi/2$ for $\Pi_n^{\mathrm{DE}}f$, and $\kappa=-3/4$ for $(\alpha,\beta)=(1/2,1)$ and $\kappa=-1$ for $(\alpha,\beta)=(1,1)$ and $\nu=\sqrt{2}\rho^{*}$ for $\Pi_n^{\mathrm{EF}}f$.}\label{fig:Exam2}
\end{figure}

\section{Hermite spectral approximation with scalings}\label{sec:ScaledHerm}
In the previous section we have established the root-exponential convergence of Hermite spectral approximation for functions with endpoint singularities without the use of scaling. In this section we consider these approximations with scalings and show that the convergence rate can be significantly improved when suitable scaling factors are adopted.

We define the functions in \eqref{def:PsiTran} with a scaling factor as
\begin{equation}\label{def:ScaledHermFun}
\Psi_{k,\lambda}(x) := (\psi_k\circ \lambda\varphi)(x), \quad  x\in I, 
\end{equation}
where $\lambda>0$ is a scaling factor and let $\mathbb{M}_{n}^{\lambda}=\mathrm{span}\{\Psi_{k,\lambda}\}_{k=0}^{n}$. The orthogonal projection of $f$ onto the space $\mathbb{M}_{n}^{\lambda}$ is given by
\begin{align}\label{def:ScaledHermProj}
(\Pi_{n,\lambda}f)(x) = \sum_{k=0}^{n} a_{k,\lambda} \Psi_{k,\lambda}(x), \quad  a_{k,\lambda} = \lambda \int_{I} f(x) \Psi_{k,\lambda}(x) \varphi'(x) dx.
\end{align}
In the following two subsections we shall consider the choice of scaling factors for each exponential transform.

\subsection{Optimal scaling of SE-Hermite approximation}
Let $\mathbb{H}_{n}^{\lambda}=\mathrm{span}\{\psi_k(\lambda x)\}_{k=0}^{n}$ and let $\Pi_{n,\lambda}^{\mathrm{SH}}$ denote the orthogonal projection operator from $L^2(\mathbb{R})$ onto the space $\mathbb{H}_{n}^{\lambda}$. Then, for any $u\in L^2(\mathbb{R})$,
\begin{equation}\label{eq:ScaledHerm}
(\Pi_{n,\lambda}^{\mathrm{SH}}u)(x) = \sum_{k=0}^{n} c_k^{\lambda} \psi_k(\lambda x),  \quad  c_k^{\lambda} = \lambda \int_{\mathbb{R}} u(x) \psi_k(\lambda x) dx.
\end{equation}
Below we state the optimal scaling factor of $\Pi_{n,\lambda}^{\mathrm{SH}}$ for analytic functions with exponential decay behaviors at infinity. This result seems to be new and might be useful in developing efficient Hermite spectral method for PDEs.
\begin{theorem}\label{thm:ScalHerm}
If $u$ is analytic in the strip $\mathcal{S}_{\rho}$ for some $\rho>0$ and $|u(x)|=O(\exp(\alpha x))$ as $x\rightarrow-\infty$ and $|u(x)|=O(\exp(-\beta x))$ as $x\rightarrow\infty$ for some $\alpha,\beta>0$. Then, the optimal scaling factor of $\Pi_{n,\lambda}^{\mathrm{SH}}u$ is $\lambda^{\ast}=\sqrt{\tau/\rho}$ with $\tau=\min\{\alpha,\beta\}$ and
\begin{equation}
\|u - \Pi_{n,\lambda^{\ast}}^{\mathrm{SH}}u\|_{L^{\infty}(\mathbb{R})} \leq \mathcal{K} \sqrt{n} \exp(-\nu\sqrt{n}),
\end{equation}
where $\nu=\sqrt{2\tau\rho}$.
\end{theorem}
\begin{proof}
Since $|\psi_k(x)|\leq1/\pi^{1/4}$ for $x\in(-\infty,\infty)$, we have
\begin{equation}
\|u - \Pi_{n,\lambda}^{\mathrm{SH}}u \|_{L^{\infty}(\mathbb{R})} \leq \frac{1}{\pi^{1/4}} \sum_{k=n+1}^{\infty} |c_k^{\lambda}|, \nonumber
\end{equation}
and thus it is necessary to analyze the decay rate of the coefficients $\{c_k^{\lambda}\}$.
We define the scaling operator $(\mathcal{A}_{\lambda}u)(x)=u(x/\lambda)$. Note that
\[
c_k^{\lambda} = \lambda \int_{\mathbb{R}} u(x) \psi_k(\lambda x) dx = \int_{\mathbb{R}} (\mathcal{A}_{\lambda}u)(x) \psi_k(x) dx.
\]
Since $u$ is analytic in the strip $\mathcal{S}_{\rho}$, $(\mathcal{A}_{\lambda}u)(x)$ is analytic in the strip $\mathcal{S}_{\lambda\rho}$. Moreover, by the decay behavior of $|u(x)|$ as $x\rightarrow\pm\infty$, we can deduce that $|(\mathcal{A}_{\lambda}u)(x)|=O(\exp(\alpha x/\lambda))$ as $x\rightarrow-\infty$ and $|(\mathcal{A}_{\lambda}u)(x)|=O(\exp(-\beta x/\lambda))$ as $x\rightarrow\infty$ and, by \cite{Hille1940}, we obtain $|c_k^{\lambda}|\leq \mathcal{K} \exp(-\nu\sqrt{k})$, where $\nu=\sqrt{2}\min\{\lambda\rho,\tau/\lambda\}$ and $\tau=\min\{\alpha,\beta\}$. Clearly, $\nu$ takes the maximum value when $\lambda\rho=\tau/\lambda$, which gives the optimal scaling factor $\lambda^{\ast}=\sqrt{\tau/\rho}$. In this case, we have $|c_k^{\lambda}|\leq \mathcal{K} \exp(-\nu\sqrt{k})$ with $\nu=\sqrt{2\tau\rho}$, and thus
\begin{equation}
\|u - \Pi_{n,\lambda^{\ast}}^{\mathrm{SH}}u \|_{L^{\infty}(\mathbb{R})} \leq \mathcal{K} \sum_{k=n+1}^{\infty} e^{-\sqrt{2\tau\rho k}} \leq \mathcal{K} \int_{n}^{\infty} e^{-\sqrt{2\tau\rho x}} dx \leq \mathcal{K} \sqrt{n} e^{-\sqrt{2\tau\rho n}}. \nonumber
\end{equation}
This ends the proof.
\end{proof}

\begin{remark}
Under the assumptions of Theorem \ref{thm:ScalHerm}, the Hermite projection $\Pi_{n,1}^{\mathrm{SH}}u$ (i.e., $\lambda=1$) converges at the rate $O(\exp(-\nu\sqrt{n}))$ with $\nu=\sqrt{2}\min\{\tau,\rho\}$. In contrast, when choosing the optimal scaling factor $\lambda^{\ast}=\sqrt{\tau/\rho}$, the Hermite projection $\Pi_{n,\lambda^{\ast}}^{\mathrm{SH}}u$ converges at the rate $O(\exp(-\nu\sqrt{n}))$ with $\nu=\sqrt{2\tau\rho}$. Since $\min\{\tau,\rho\}\leq\sqrt{\tau\rho}$ for any $\tau,\rho>0$ with equality only when $\tau=\rho$, we conclude that the convergence rate of $\Pi_{n,\lambda^{\ast}}^{\mathrm{SH}}u$ will be faster than the Hermite projection $\Pi_{n,1}^{\mathrm{SH}}u$ when $\tau\neq\rho$.
\end{remark}

\begin{example}
Consider the function $u(x)=\mathrm{sech}(\kappa x)$ with $\kappa>0$. Clearly, its poles which are nearest to the real axis are $\pm\mathrm{i}\pi/(2\kappa)$, and thus $\rho=\pi/(2\kappa)-\epsilon$ with $\epsilon$ being arbitrarily close to zero. Moreover, $|u(x)|=O(\exp(-\kappa|x|))$ as $x\rightarrow\pm\infty$, and thus the convergence rate of the standard Hermite projection $\Pi_{n,1}^{\mathrm{SH}}u$ is $O(\exp(-\nu\sqrt{n}))$ with $\nu=\sqrt{2}\min\{\kappa,\pi/(2\kappa)\}$. In contrast, if we choose the optimal scaling factor $\lambda^{\ast}=(2/\pi)^{1/2}\kappa$, then the convergence rate of the scaled Hermite projection $\Pi_{n,\lambda^{\ast}}^{\mathrm{SH}}u$ is $O(\exp(-\nu\sqrt{n}))$ with $\nu=\sqrt{\pi}$.
\end{example}

Now we turn to the approximation \eqref{def:ScaledHermProj} with $\varphi=\varphi_{\mathrm{SE}}$. For clearness of presentation, we denote it by
\begin{equation}\label{def:ScaledSEProj}
(\Pi_{n,\lambda}^{\mathrm{SE}}f)(x) = \sum_{k=0}^{n} a_{k,\lambda}^{\mathrm{SE}} \Psi_{k,\lambda}^{\mathrm{SE}}(x), \quad  a_{k,\lambda}^{\mathrm{SE}} = \lambda \int_{I} f(x) \Psi_{k,\lambda}^{\mathrm{SE}}(x) \varphi'_{\mathrm{SE}}(x) dx.
\end{equation}
The optimal scaling factor and the corresponding convergence rate are stated in the following theorem.
\begin{theorem}\label{thm:ScaledSEConv}
If $f$ is analytic in $\phi_{\mathrm{SE}}(\mathcal{S}_{\rho})$ for some $\rho\in(0,\pi)$ and $|f(x)|\leq \mathcal{K} |(x-a)^{\alpha}(b-x)^{\beta}|$ for some $\alpha,\beta>0$ and $x\in I$. Then, the optimal scaling factor of $\Pi_{n,\lambda}^{\mathrm{SE}}f$ is $\lambda^{\ast}=\sqrt{\tau/\rho}$ with $\tau=\min\{\alpha,\beta\}$, and
\begin{equation}\label{eq:ScaledSEConv}
\|f - \Pi_{n,\lambda^{\ast}}^{\mathrm{SE}}f\|_{L^{\infty}(I)} \leq \mathcal{K} \sqrt{n} \exp(-\nu\sqrt{n}),
\end{equation}
where $\nu=\sqrt{2\tau\rho}$.
\end{theorem}
\begin{proof}
Let $x=\phi_{\mathrm{SE}}(s)$ with $s\in(-\infty,\infty)$ and $g(s)=(f\circ\phi_{\mathrm{SE}})(s)$. Clearly, $(\Pi_{n,\lambda}^{\mathrm{SE}}f)(x) = (\Pi_{n,\lambda}^{\mathrm{SH}}g)(s)$ and
\begin{equation}
\|f - \Pi_{n,\lambda}^{\mathrm{SE}}f\|_{L^{\infty}(I)} = \|g - \Pi_{n,\lambda}^{\mathrm{SH}}g \|_{L^{\infty}(\mathbb{R})}. \nonumber
\end{equation}
By the assumptions on $f$, it is easily verified that $g(s)$ is analytic in the strip $\mathcal{S}_{\rho}$ and
$|g(s)|=O(\exp(\alpha s))$ as $s\rightarrow-\infty$ and $|g(s)|=O(\exp(-\beta s))$ as $s\rightarrow\infty$, the desired result \eqref{eq:ScaledSEConv} follows immediately from Theorem \ref{thm:ScalHerm}. This ends the proof.
\end{proof}

Now we present numerical experiments verifying Theorem \ref{thm:ScaledSEConv}. We consider $f(x) = x^{\alpha}(1-x)^{\beta}$ with $(\alpha,\beta)=(1/2,1/2)$ and $(\alpha,\beta)=(2/3,1)$ and $f(x)=x^{\alpha}\log^{\beta}(x)$ with $(\alpha,\beta)=(1,1)$ and $(\alpha,\beta)=(1,2)$. In Figure \ref{fig:Exam3} we plot the maximum errors of $\Pi_{n,\lambda}^{\mathrm{SE}}f$ with $\lambda=\lambda^{\ast}$ and $\lambda=1$ as a function of $\sqrt{n}$. Moreover, we also plot the predicted convergence rates $O(n^{\kappa}\exp(-\nu\sqrt{n}))$ with $\nu=\sqrt{2\tau\pi}$ for $\lambda=\lambda^{\ast}$ and $\nu=\sqrt{2}\min\{\tau,\pi\}$ for $\lambda=1$. We see that the maximum error of $\Pi_{n,\lambda^{\ast}}^{\mathrm{SE}}f$ decays at the predicted rate, up to some small algebraic factors. Moreover, we also see that the convergence rate of $\Pi_{n,\lambda^{\ast}}^{\mathrm{SE}}f$ is much faster than that of $\Pi_{n,1}^{\mathrm{SE}}f$.

\begin{figure}[htbp]
\centering
\includegraphics[width=0.49\textwidth,height=0.42\textwidth]{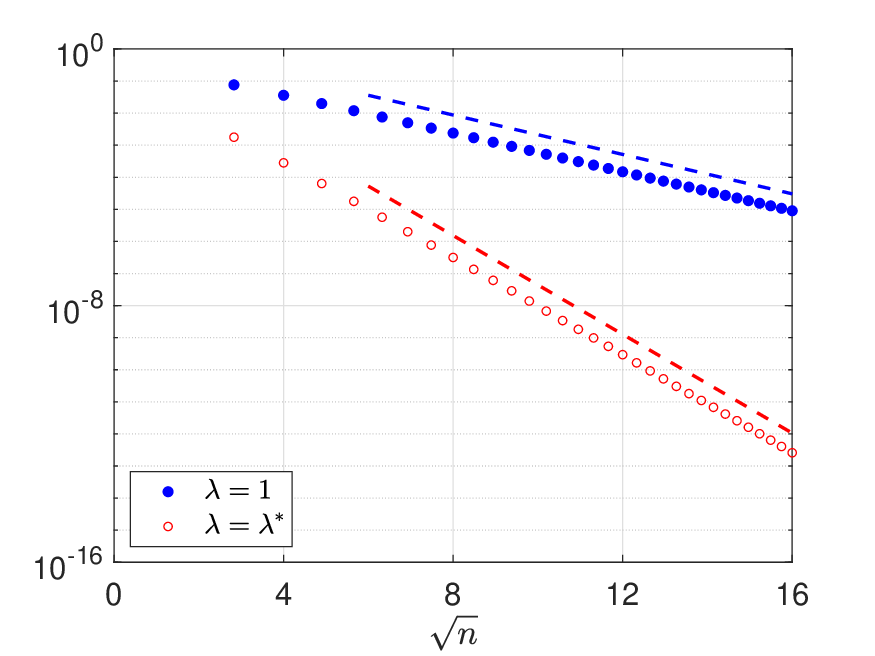}
\includegraphics[width=0.49\textwidth,height=0.42\textwidth]{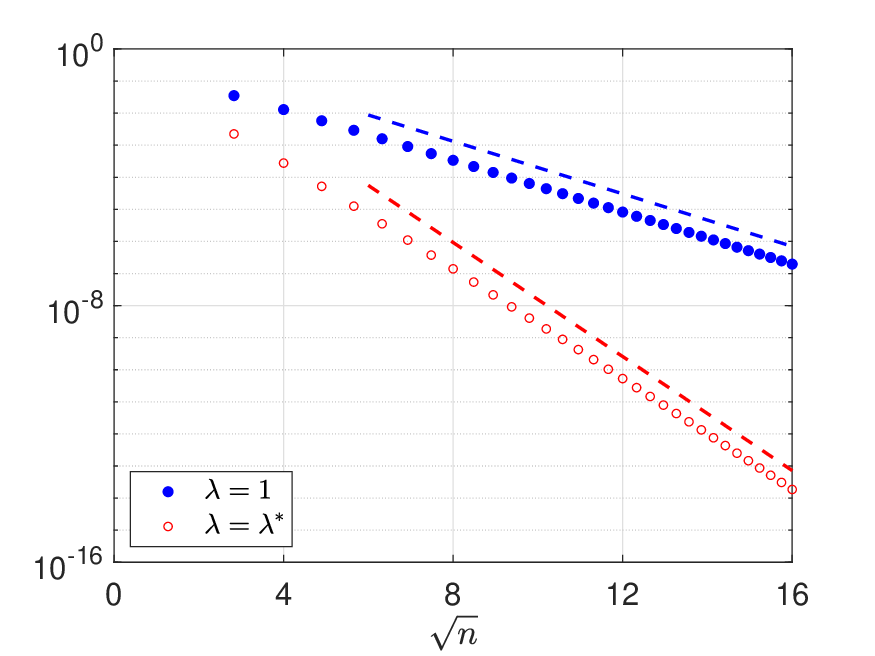}\\
\includegraphics[width=0.49\textwidth,height=0.42\textwidth]{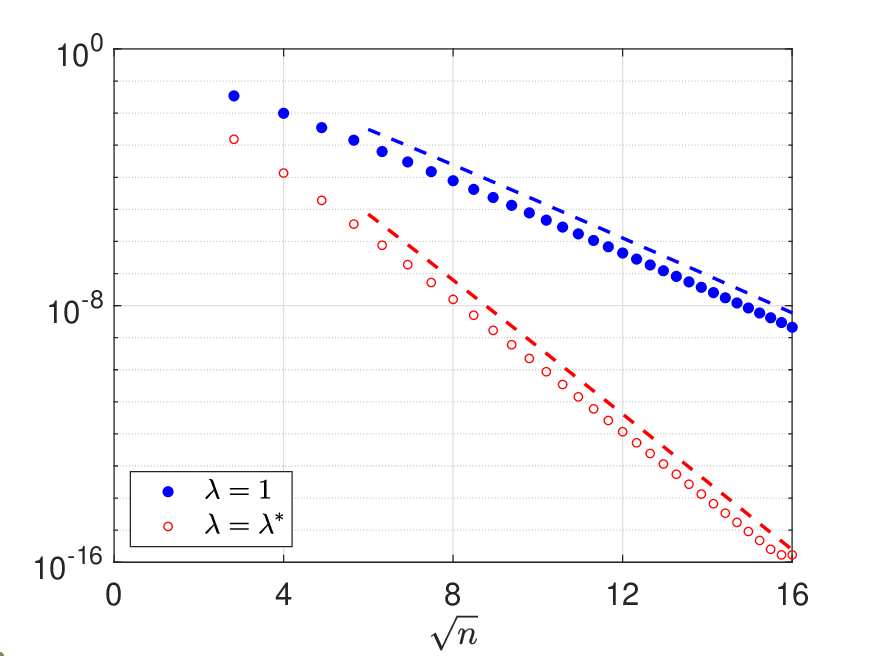}
\includegraphics[width=0.49\textwidth,height=0.42\textwidth]{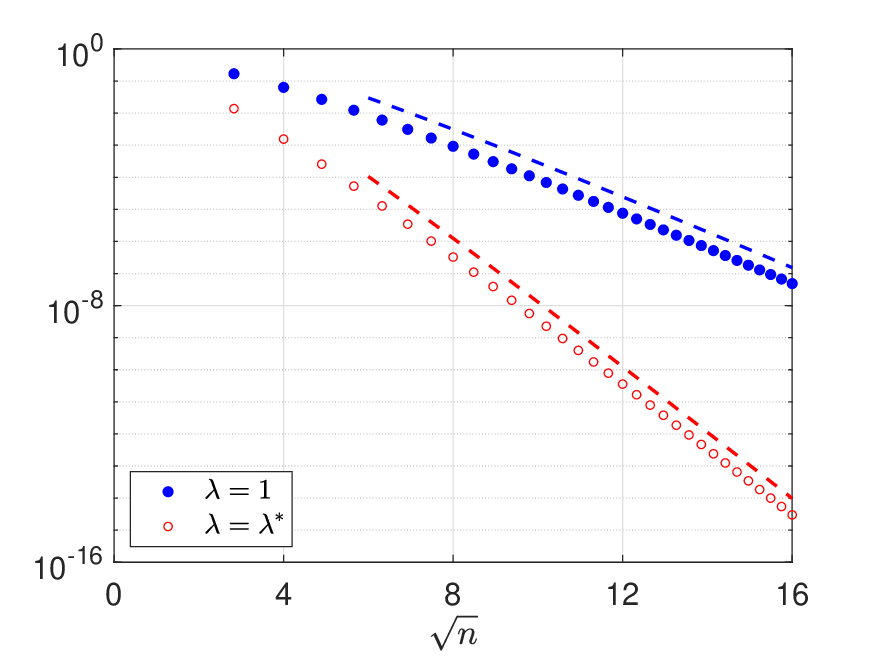}
\caption{Top: Maximum errors of $\Pi_{n,\lambda}^{\mathrm{SE}}f$ with $\lambda=\lambda^{\ast}$ and $\lambda=1$ as a function of $\sqrt{n}$ for $f(x) = x^{\alpha}(1-x)^{\beta}$. Here $(\alpha,\beta)=(1/2,1/2)$ (left) and $(\alpha,\beta)=(2/3,1)$ (right). Bottom: Maximum errors of $\Pi_{n,\lambda}^{\mathrm{SE}}f$ with $\lambda=\lambda^{\ast}$ and $\lambda=1$ as a function of $\sqrt{n}$ for $f(x)=x^{\alpha}\log^{\beta}(x)$. Here $(\alpha,\beta)=(1,1)$ (left) and $(\alpha,\beta)=(1,2)$ (right). The dashed lines from top to bottom show the rates $O(n^{\kappa}\exp(-\nu\sqrt{n}))$ with $\nu=\sqrt{2}\min\{\tau,\pi\}$ and $\nu=\sqrt{2\tau\pi}$, and $\kappa=0$ (top row) and $\kappa=1/2$ (bottom, left) and $\kappa=1$ (bottom, right).
}\label{fig:Exam3}
\end{figure}

SE-Sinc approximation is an efficient method for approximating functions with endpoint singularities \cite{Stenger1993}. Let $S(j,h)(x)$ denote the sinc function
\begin{equation}\label{def:SincFun}
S(j,h)(x) = \frac{\sin(\pi(x/h-j))}{\pi(x/h-j)},
\end{equation}
where $h>0$ is the step size. The SE-Sinc approximation is defined by
\begin{align}\label{def:SESinc}
(\mathcal{T}_{M+N}^{\mathrm{SE}}f)(x) &= \sum_{j=-M}^{N} f(\phi_{\mathrm{SE}}(jh)) S(j,h)(\varphi_{\mathrm{SE}}(x)),
\end{align}
where $M=n$ and $N=\lfloor \alpha n/\beta \rfloor$ if $\tau=\alpha$, and $N=n$ and $M=\lfloor \beta n/\alpha \rfloor$ if $\tau=\beta$. Under the assumptions of Theorem \ref{thm:ScaledSEConv} and choosing $h=\sqrt{\pi\rho/(\tau n)}$, by \cite[Theorem~4.2.5]{Stenger1993} we know that
\begin{align}\label{eq:SESincBound}
\|f - \mathcal{T}_{M+N}^{\mathrm{SE}}f \|_{L^{\infty}(I)} \leq \mathcal{K} \sqrt{n} \exp(-\nu\sqrt{n}),
\end{align}
where $\nu=\sqrt{\pi\tau\rho}$. Regarding the SE-Sinc and SE-Hermite approximations using the same number of terms, i.e., $(\mathcal{T}_{M+N}^{\mathrm{SE}}f)(x)$ and $(\Pi_{M+N,\lambda^{\ast}}^{\mathrm{SE}}f)(x)$, which one has faster convergence rate? Below we state the comparison result:
\begin{itemize}
\item Let $\sigma=\min\{\alpha,\beta\}/\max\{\alpha,\beta\}+1$, then $M+N\approx\sigma n$. By Theorem \ref{thm:ScaledSEConv} we know that $\Pi_{M+N,\lambda^{\ast}}^{\mathrm{SE}}f$ converges at the rate $O(\exp(-\nu\sqrt{n}))$ with $\nu=\sqrt{2\sigma\tau\rho}$. Comparing this with \eqref{eq:SESincBound}, we see that both methods converge at the same rate when $\sigma=\pi/2\approx1.5708$, and the convergence rate of $\Pi_{M+N,\lambda^{\ast}}^{\mathrm{SE}}f$ will be faster when $\sigma\in(\pi/2,2]$ and the convergence rate of $\mathcal{T}_{M+N}^{\mathrm{SE}}f$ will be faster when $\sigma\in(1,\pi/2)$.
\end{itemize}

\begin{figure}[htbp]
\centering
\includegraphics[width=0.49\textwidth,height=0.42\textwidth]{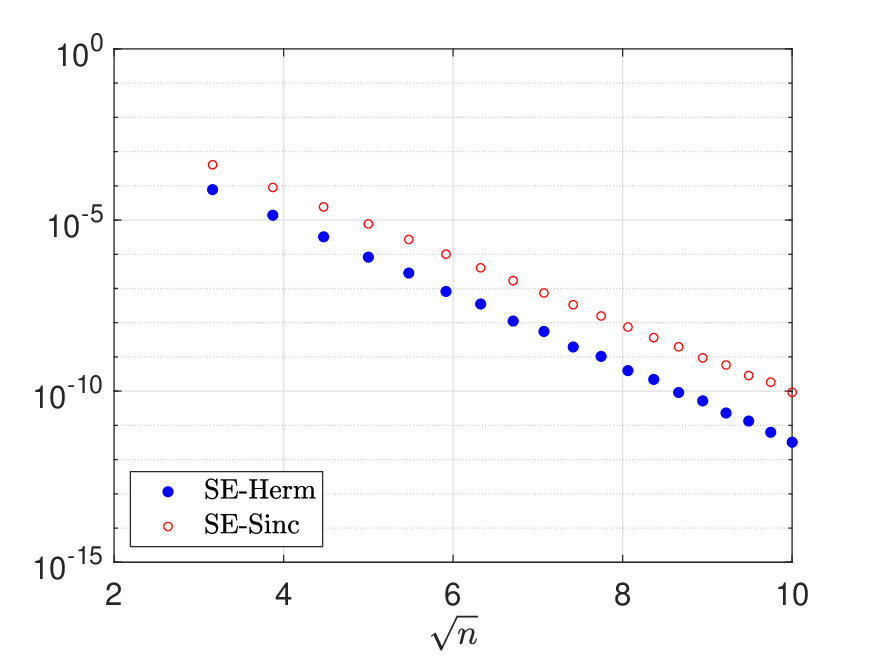}
\includegraphics[width=0.49\textwidth,height=0.42\textwidth]{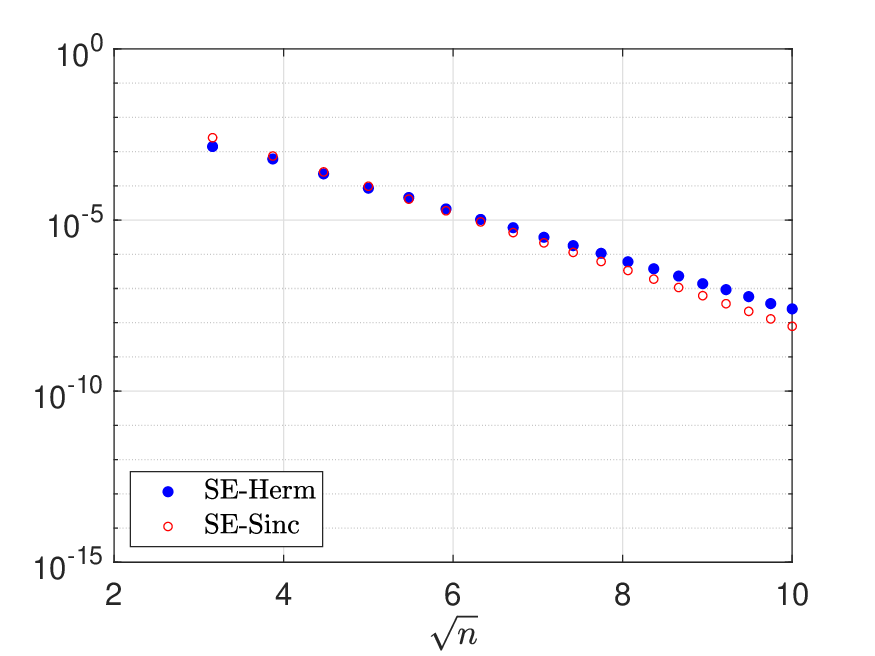}\\
\includegraphics[width=0.49\textwidth,height=0.42\textwidth]{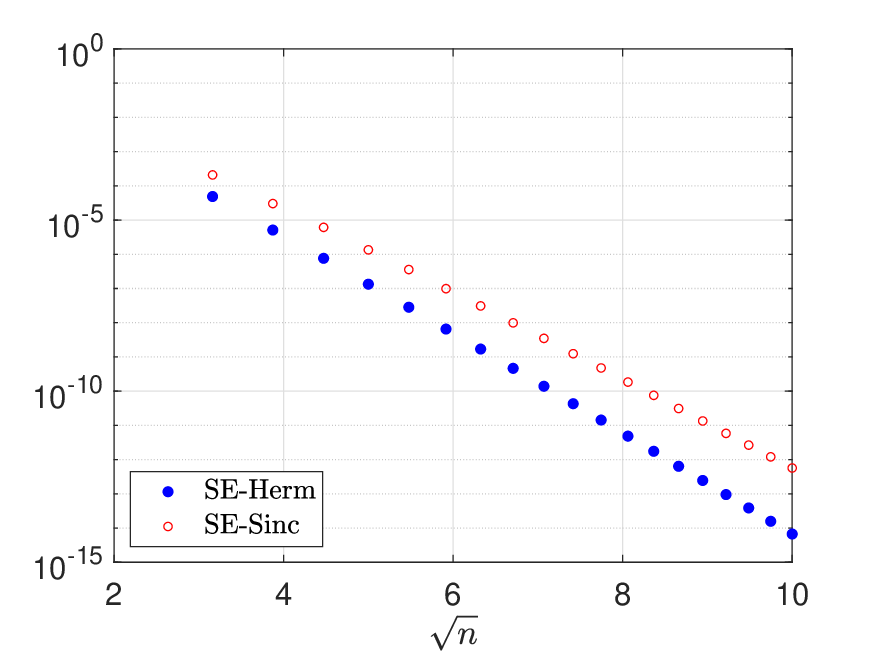}
\includegraphics[width=0.49\textwidth,height=0.42\textwidth]{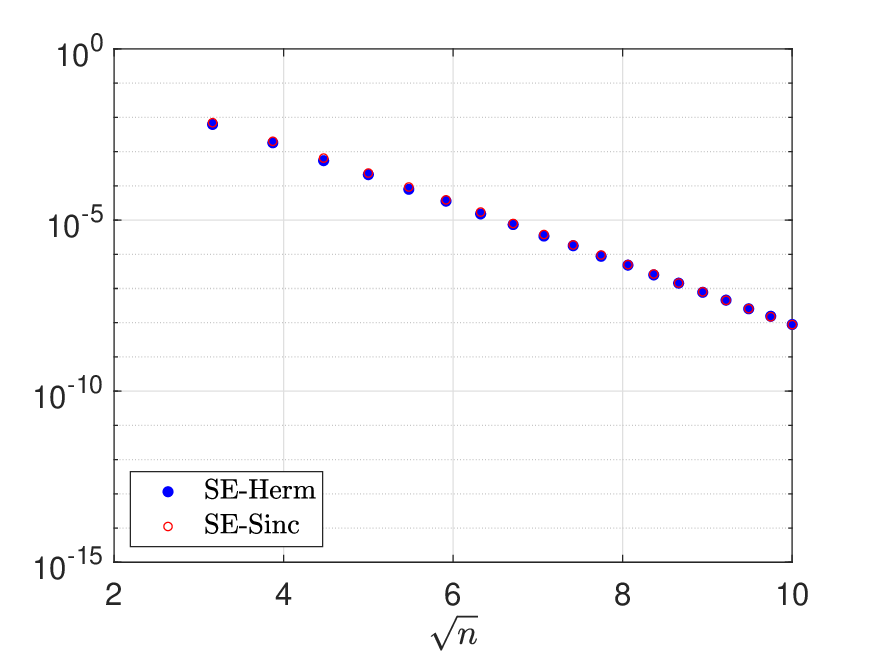}
\caption{Top: Maximum errors of $\mathcal{T}_{M+N}^{\mathrm{SE}}f$ (SE-Sinc) and $\Pi_{M+N,\lambda^{\ast}}^{\mathrm{SE}}f$ (SE-Herm) as a function of $\sqrt{n}$ for $f(x)=x^{\alpha}(1-x)^{\beta}/(x+2)$ with $(\alpha,\beta)=(1/2,1/2)$ (left) and $(\alpha,\beta)=(1/3,1)$ (right). Bottom: Maximum errors of $\mathcal{T}_{M+N}^{\mathrm{SE}}f$ (SE-Sinc) and $\Pi_{M+N,\lambda^{\ast}}^{\mathrm{SE}}f$ (SE-Herm) as a function of $\sqrt{n}$ for $f(x)=x^{\alpha}\log^{\beta}(x)$ with $(\alpha,\beta)=(1,1)$ (left) and $(\alpha,\beta)=(1/2,1)$ (right). }\label{fig:Exam4}
\end{figure}

In the top row of Figure \ref{fig:Exam4} we plot the maximum errors of $\mathcal{T}_{M+N}^{\mathrm{SE}}f$ and $\Pi_{M+N,\lambda^{\ast}}^{\mathrm{SE}}f$ as a function of $\sqrt{n}$ for $f(x)=x^{\alpha}(1-x)^{\beta}g(x)$ with $(\alpha,\beta)=(1/2,1/2)$ and $(\alpha,\beta)=(1/2,3/2)$ and $g(x)=1/(x+2)$. The step size in $\mathcal{T}_{M+N}^{\mathrm{SE}}f$ is chosen as $h=\pi/\sqrt{\tau n}$. As expected, we see that the convergence rate of $\Pi_{M+N,\lambda^{\ast}}^{\mathrm{SE}}f$ is obviously faster than that of $\mathcal{T}_{M+N}^{\mathrm{SE}}f$ for $(\alpha,\beta)=(1/2,1/2)$ and the convergence rate of $\mathcal{T}_{M+N}^{\mathrm{SE}}f$ is slightly faster than that of $\Pi_{M+N,\lambda^{\ast}}^{\mathrm{SE}}f$ for $(\alpha,\beta)=(1/2,3/2)$. In the bottom row of Figure \ref{fig:Exam4} we plot the maximum errors of $\mathcal{T}_{M+N}^{\mathrm{SE}}f$ and $\Pi_{M+N,\lambda^{\ast}}^{\mathrm{SE}}f$ as a function of $\sqrt{n}$ for $f(x)=x^{\alpha}\log^{\beta}(x)$ with $(\alpha,\beta)=(1,1)$ and $(\alpha,\beta)=(1/2,1)$. We see that the convergence rate of $\Pi_{M+N,\lambda^{\ast}}^{\mathrm{SE}}f$ is obviously faster than that of $\mathcal{T}_{M+N}^{\mathrm{SE}}f$ for $(\alpha,\beta)=(1,1)$ and both are indistinguishable for $(\alpha,\beta)=(1/2,1)$.

\subsection{Optimal scaling of DE- and EF-Hermite approximations}
In this subsection we consider the approximation \eqref{def:ScaledHermProj} with $\varphi=\varphi_{\mathrm{DE}}$ and $\varphi=\varphi_{\mathrm{EF}}$ and, for clearness of presentation, we denote them respectively by
\begin{equation}\label{def:ScaledDEProj}
(\Pi_{n,\lambda}^{\mathrm{DE}}f)(x) = \sum_{k=0}^{n} a_{k,\lambda}^{\mathrm{DE}} \Psi_{k,\lambda}^{\mathrm{DE}}(x), \quad  a_{k,\lambda}^{\mathrm{DE}} = \lambda \int_{I} f(x) \Psi_{k,\lambda}^{\mathrm{DE}}(x) \varphi'_{\mathrm{DE}}(x) dx,
\end{equation}
and
\begin{equation}\label{def:ScaledEFProj}
(\Pi_{n,\lambda}^{\mathrm{EF}}f)(x) = \sum_{k=0}^{n} a_{k,\lambda}^{\mathrm{EF}} \Psi_{k,\lambda}^{\mathrm{EF}}(x), \quad  a_{k,\lambda}^{\mathrm{EF}} = \lambda \int_{I} f(x) \Psi_{k,\lambda}^{\mathrm{EF}}(x) \varphi'_{\mathrm{EF}}(x) dx.
\end{equation}
We begin by recalling a recent result on scaling factor of the approximation \eqref{eq:ScaledHerm}. In the sequel $f\lesssim g$ means that there exists a fixed constant $\mathcal{K}>0$ such that $f\leq \mathcal{K} g$.
\begin{theorem}\label{thm:ScaledHerm}
Let $u\in L^2(\mathbb{R})$ and let $\Pi_{n,\lambda}^{\mathrm{SH}}u$ denote the Hermite approximation defined in \eqref{eq:ScaledHerm}. Then, it holds
\begin{equation}
\|u - \Pi_{n,\lambda}^{\mathrm{SH}}u \|_{L^2(\mathbb{R})} \lesssim \|u(x)\cdot\mathbb{I}_{\{|x|>M/\lambda\}} \|_{L^2(\mathbb{R})} + \|\hat{u}(\xi) \cdot \mathbb{I}_{\{|\xi|>\lambda M\}} \|_{L^2(\mathbb{R})} + e^{-n/16} \|u\|_{L^2(\mathbb{R})}, \nonumber
\end{equation}
where $\hat{u}(\cdot)=\mathcal{F}[u](\cdot)$ is the Fourier transform of $u$ and $M=\sqrt{n}/(2\sqrt{2})$.
\end{theorem}
\begin{proof}
See \cite[Theorem~2.4]{Hu2026}.
\end{proof}

Below we present a generalization of the above theorem to $L^{\infty}$-norm.
\begin{theorem}\label{thm:ScaledHermInf}
Let $M=\sigma\sqrt{n}$ and $N=\eta\sqrt{n}$ with $\sigma,\eta>0$. Then, for $4\sigma^2+\eta^2\in(0,2)$, it holds
\begin{align}
\|u - \Pi_{n,\lambda}^{\mathrm{SH}}u \|_{L^{\infty}(\mathbb{R})} &\lesssim (1+\Lambda_n) \left( \|u(x)\cdot\mathbb{I}_{\{|x|>M/\lambda\}} \|_{L^{\infty}(\mathbb{R})} + \|\hat{u}(\xi) \cdot \mathbb{I}_{\{|\xi|>\lambda N\}} \|_{L^1(\mathbb{R})} \right) \nonumber \\[1ex]
&+ e^{-\sigma^2 n} \left( \frac{1+\Lambda_n}{\sigma\sqrt{n}} \right) \|u\|_{L^{\infty}(\mathbb{R})} + \lambda n^{3/4} e^{-\varrho n} \|\hat{u}\|_{L^{\infty}(\mathbb{R})} , \nonumber
\end{align}
where $\Lambda_n$ is the Lebesgue constant of the projection operator $\Pi_{n,1}^{\mathrm{SH}}$ in $L^{\infty}$-norm, i.e.,
\[
\Lambda_n = \sup_{f\not\equiv0} \frac{\|\Pi_{n,1}^{\mathrm{SH}}f\|_{L^{\infty}(\mathbb{R})}}{\|f\|_{L^{\infty}(\mathbb{R})}},
\]
and
\[
\varrho = \frac{4\sigma^2+\eta^2}{4} - \frac{1}{2}\log\frac{e(4\sigma^2+\eta^2)}{2}.
\]
\end{theorem}
\begin{proof}
The proof is similar to that of \cite[Theorem~2.4]{Hu2026} and we defer it to Appendix A for clarity of exposition.
\end{proof}

Now we turn to DE- and EF-Hermite approximations in \eqref{def:ScaledDEProj} and \eqref{def:ScaledEFProj} and show that their convergence rate can be significantly improved when suitable scaling factors are chosen. We state the results in the following theorem.
\begin{theorem}\label{thm:ScaledDE}
The following results hold.
\begin{itemize}
\item[\rm(i)] If $f$ is analytic on $\phi_{\mathrm{DE}}(\mathcal{S}_{\rho})$ for some $\rho\in(0,\pi/2)$ and $|f(x)| \leq \mathcal{K} |(x-a)^{\alpha}(b-x)^{\beta}|$ for some $\alpha,\beta>0$ and $x\in I$. When choosing $\lambda=\kappa\sqrt{n}/\log n$ with $\kappa>0$, then
\begin{equation}\label{eq:ScaledDEConv}
\|f - \Pi_{n,\lambda}^{\mathrm{DE}}f\|_{L^{\infty}(I)} \leq \mathcal{K} \exp\left(-\nu n/\log n\right),
\end{equation}
for some $\nu>0$.

\item[\rm(ii)] If $f$ is analytic on $\phi_{\mathrm{EF}}(\mathcal{S}_{\rho})$ for some $\rho\in(0,\rho^{*})$ and $|f(x)|\leq \mathcal{K} |(x-a)^{\alpha}(b-x)^{\beta}|$ for some $\alpha,\beta>0$ and $x\in I$. When choosing $\lambda=\kappa n^{1/6}$ with $\kappa>0$, then
\begin{equation}\label{eq:ScaledEFConv}
\|f - \Pi_{n,\lambda}^{\mathrm{EF}}f\|_{L^{\infty}(I)} \leq \mathcal{K} \exp(-\nu n^{2/3}),
\end{equation}
for some $\nu>0$.
\end{itemize}
\end{theorem}
\begin{proof}
We first consider the case of $\Pi_{n,\lambda}^{\mathrm{DE}}f$.
Let $x=\phi_{\mathrm{DE}}(s)$ and let $F_D(s)=(f\circ\phi_{\mathrm{DE}})(s)$ with $s\in(-\infty,\infty)$. It is easily checked that $(\Pi_{n,\lambda}^{\mathrm{DE}}f)(x) = (\Pi_{n,\lambda}^{\mathrm{SH}}F_D)(s)$, and thus
\begin{equation}
\|f - \Pi_{n,\lambda}^{\mathrm{DE}}f\|_{L^{\infty}(I)} = \|F_D - \Pi_{n,\lambda}^{\mathrm{SH}}F_D\|_{L^{\infty}(\mathbb{R})}. \nonumber
\end{equation}
Since $f$ is analytic on $\phi_{\mathrm{DE}}(\mathcal{S}_{\rho})$ and $|f(x)|\leq\mathcal{K}|(x-a)^{\alpha}(b-x)^{\beta}|$ for $x\in{I}$, it follows that $F_D(s)$ is analytic on the strip $\mathcal{S}_{\rho}$ and for $x+\mathrm{i}y\in\mathcal{S}_{\rho}$,
\begin{align}
|F_D(x+\mathrm{i}y)| = O\left(\exp\left(-\frac{\tau\pi}{2} e^{|x|}\cos(y)\right)\right), \nonumber
\end{align}
as $x\rightarrow\pm\infty$, where $\tau=\alpha$ when $x\rightarrow-\infty$ and $\tau=\beta$ when $x\rightarrow\infty$. Hence we see that $|F_D(x+\mathrm{i}y)|$ decays double-exponentially as $x\rightarrow\pm\infty$, and thus
\[
\int_{\partial\mathcal{S}_{\rho}} |F_D(z)| |dz| < \infty,  \qquad  \lim_{x\rightarrow\pm\infty} \sup_{|y|\leq\rho} |F_D(x+\mathrm{i}y)| = 0.
\]
By \cite[Lemma~4.7]{Comte2025} we know that $\mathcal{F}[F_D](\xi)$ decays at the exponential rate $O(\exp(-\rho|\xi|))$ as $|\xi|\rightarrow\infty$. According to Theorem \ref{thm:ScaledHermInf}, the optimal scaling factor can be derived by balancing the decay rates of $\|F_D(s)\cdot\mathbb{I}_{\{|s|>M/\lambda\}}\|_{L^{\infty}(\mathbb{R})}$ and $\|\mathcal{F}[F_D](\xi) \cdot \mathbb{I}_{\{|\xi|>\lambda N\}} \|_{L^{1}(\mathbb{R})}$, we obtain
\begin{align}
\exp\left(-\frac{\tau\pi}{2} e^{M/\lambda}\right) = \exp(-\rho\lambda{N}) , \nonumber
\end{align}
which gives
\[
\lambda = M \left( W\left(\frac{2\rho{M}{N}}{\tau\pi} \right) \right)^{-1} = \sigma\sqrt{n} \left( W\left(\frac{2\rho\sigma\eta n}{\tau\pi}\right) \right)^{-1},
\]
where $W(z)$ is the Lambert $W$-function. Since $W(x)\sim\log x$ as $x\rightarrow\infty$, we deduce that the optimal scaling factor asymptotes to $\lambda=O(\sqrt{n}/\log n)$ as $n\rightarrow\infty$. In this case, the error of $\Pi_{n,\lambda}^{\mathrm{DE}}f$ decays at the almost-exponential rate $O(\exp(-\nu n/\log n))$ for some $\nu>0$ as $n\rightarrow\infty$. This proves (i).

As for the case of $\Pi_{n,\lambda}^{\mathrm{EF}}f$, we proceed in a similar way. Let $x=\phi_{\mathrm{EF}}(s)$ and let $F_E(s)=(f\circ\phi_{\mathrm{EF}})(s)$, then we have $(\Pi_{n,\lambda}^{\mathrm{EF}}f)(x) = (\Pi_{n,\lambda}^{\mathrm{SH}}F_E)(s)$ and
\[
\|f - \Pi_{n,\lambda}^{\mathrm{EF}}f\|_{L^{\infty}(I)} = \|F_E - \Pi_{n,\lambda}^{\mathrm{SH}}F_E\|_{L^{\infty}(\mathbb{R})}.
\]
Since $f$ is analytic on $\phi_{\mathrm{EF}}(\mathcal{S}_{\rho})$ and $|f(x)|\leq\mathcal{K}|(x-a)^{\alpha}(b-x)^{\beta}|$ for $x\in{I}$, it follows that $F_E(z)$ is analytic on the strip $\mathcal{S}_{\rho}$ and for $x+\mathrm{i}y\in\mathcal{S}_{\rho}$,
\begin{align}
|F_E(x+\mathrm{i}y)| = O\left(\frac{\exp(-\tau(x^2-y^2))}{(x^2+y^2)^{\tau/2}}\right), \nonumber
\end{align}
as $x\rightarrow\pm\infty$, where $\tau=\alpha$ when $x\rightarrow-\infty$ and $\tau=\beta$ when $x\rightarrow\infty$. Hence we see that $|F_E(x+\mathrm{i}y)|$ has Gaussian decay as $x\rightarrow\pm\infty$, and thus
\[
\int_{\partial\mathcal{S}_{\rho}} |F_E(z)| |dz| < \infty,  \qquad  \lim_{x\rightarrow\pm\infty} \sup_{|y|\leq\rho} |F_E(x+\mathrm{i}y)| = 0.
\]
By \cite[Lemma~4.7]{Comte2025} again we know that $\mathcal{F}[F_E](\xi)$ decays at the exponential rate $O(\exp(-\rho|\xi|))$ as $|\xi|\rightarrow\infty$. Moreover, using Theorem \ref{thm:ScaledHermInf} again and balancing the decay rates of $\|F_E(s)\cdot\mathbb{I}_{\{|s|>M/\lambda\}}\|_{L^{\infty}(\mathbb{R})}$ and $\|\mathcal{F}[F_E](\xi) \cdot \mathbb{I}_{\{|\xi|>\lambda{N}\}} \|_{L^{1}(\mathbb{R})}$ we obtain
\begin{align}
\exp\left(-\frac{\tau M^2}{\lambda^2}\right) = \exp(-\rho\lambda N),  \nonumber
\end{align}
which gives
\[
\lambda = \left(\frac{\tau{M}^2}{\rho{N}}\right)^{1/3} = \left(\frac{\tau\sigma^2}{\rho\eta} \right)^{1/3} n^{1/6} .
\]
Thus the optimal scaling factor satisfies $\lambda=O(n^{1/6})$ as $n\rightarrow\infty$. In this case, the error of $\Pi_{n,\lambda}^{\mathrm{EF}}f$ decays at the rate $O(\exp(-\nu n^{2/3}))$ for some $\nu>0$ as $n\rightarrow\infty$. This ends the proof.
\end{proof}

\begin{figure}[htbp]
\centering
\includegraphics[width=0.49\textwidth,height=0.42\textwidth]{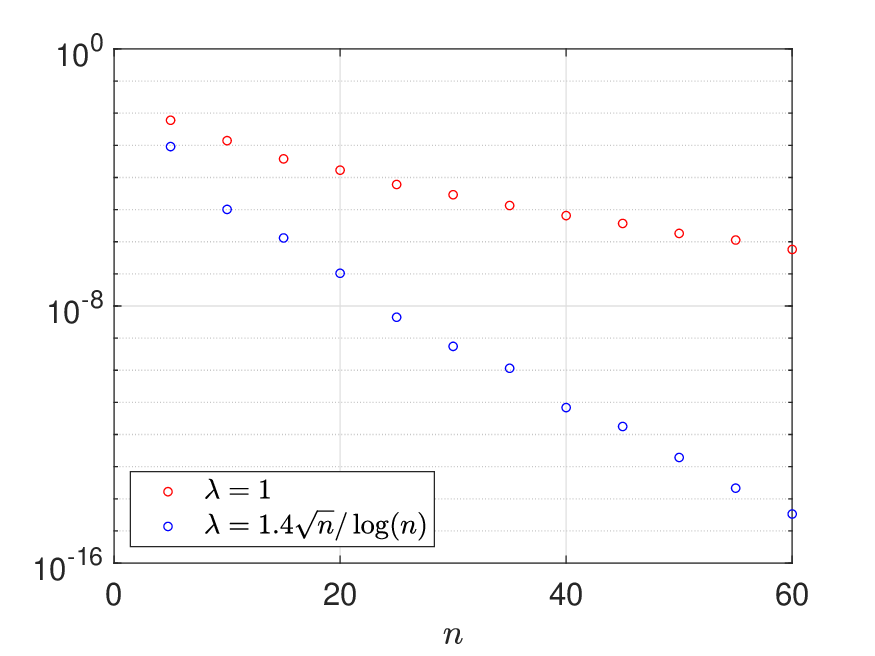}
\includegraphics[width=0.49\textwidth,height=0.42\textwidth]{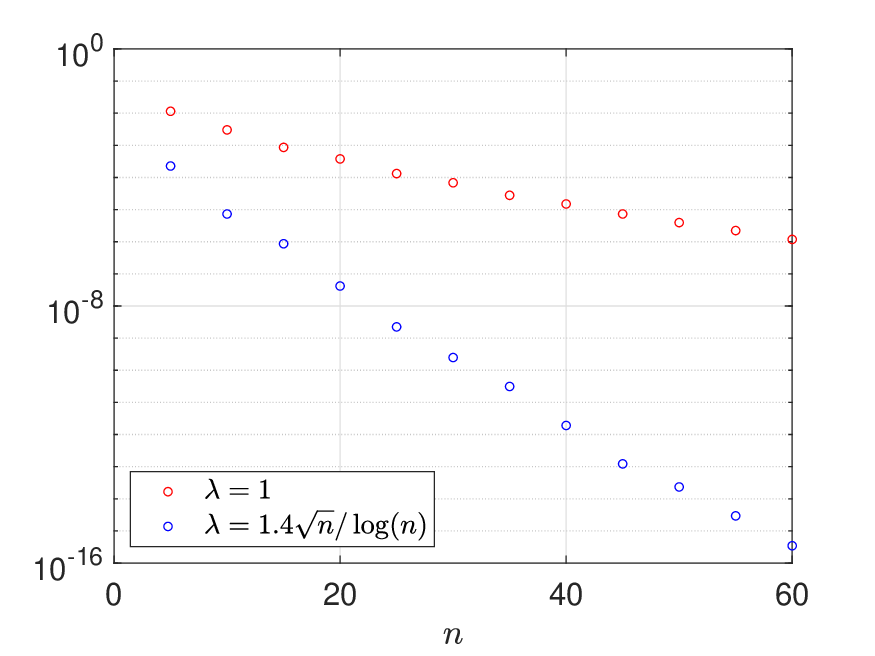}\\
\includegraphics[width=0.49\textwidth,height=0.42\textwidth]{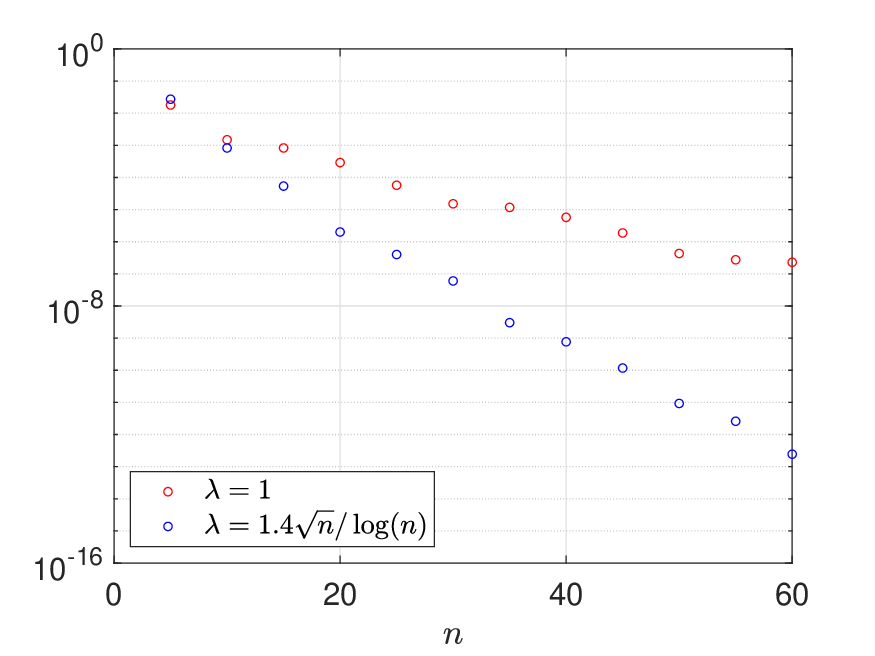}
\includegraphics[width=0.49\textwidth,height=0.42\textwidth]{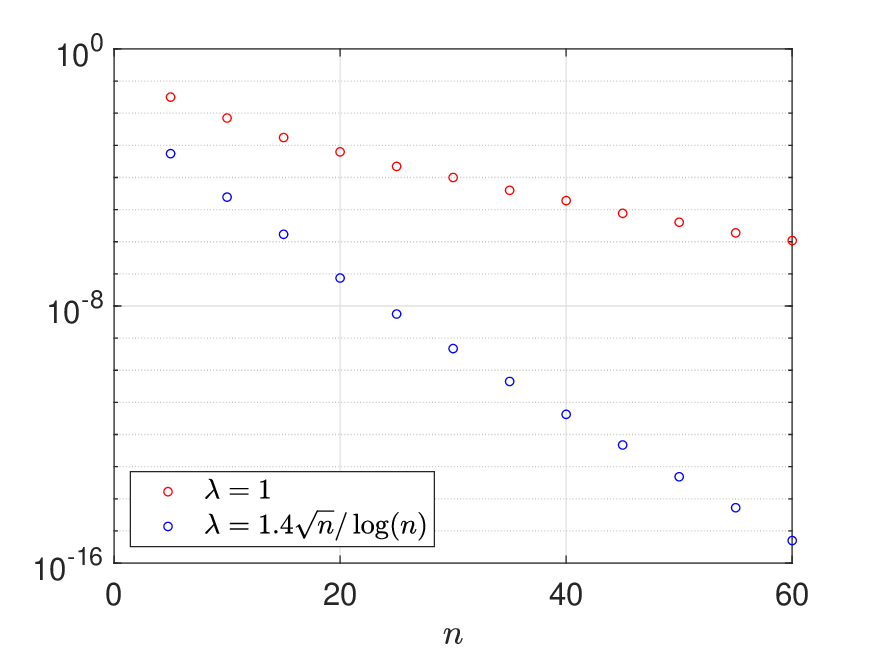}
\caption{Top: Maximum errors of $\Pi_{n,\lambda}^{\mathrm{DE}}f$ as a function of $n$ for $f(x) = x^{\alpha}(1-x)^{\beta}g(x)$ with $g(x)=1/(x+2)$ and $(\alpha,\beta)=(1/2,1/2)$ (left) and $(\alpha,\beta)=(2/3,1)$ (right). Bottom: Maximum errors of $\Pi_{n,\lambda}^{\mathrm{DE}}f$ as a function of $n$ for $f(x) = x^{\alpha}\log^{\beta}(x)$ with $(\alpha,\beta)=(1/2,1)$ (left) and $(\alpha,\beta)=(1,1)$ (right).
}\label{fig:Exam5}
\end{figure}

To illustrate the performance of the scaling factor derived in Theorem \ref{thm:ScaledDE}, we plot in Figure \ref{fig:Exam5} the maximum errors of $\Pi_{n,\lambda}^{\mathrm{DE}}f$ as a function of $n$ for $f(x) = x^{\alpha}(1-x)^{\beta}g(x)$ with $g(x)=1/(x+2)$ and $(\alpha,\beta)=(1/2,1/2)$ and $(\alpha,\beta)=(2/3,1)$. The scaling factor is chosen as $\lambda=\kappa\sqrt{n}/\log n$ with $\kappa=1.4$. For comparison, we also add the maximum errors of the approximation $\Pi_{n}^{\mathrm{DE}}f$ (i.e., $\Pi_{n,\lambda}^{\mathrm{DE}}f$ with $\lambda=1$). As expected, we see that the maximum errors of $\Pi_{n,\lambda}^{\mathrm{DE}}f$ with $\lambda=\kappa\sqrt{n}/\log n$ decay at almost-exponential rate, which is much faster than that of $\Pi_{n}^{\mathrm{DE}}f$.

Finally, we consider a comparison of DE-Hermite and DE-Sinc approximations. The DE-Sinc approximation is defined by
\begin{align}\label{def:DESinc}
(\mathcal{T}_{2n}^{\mathrm{DE}}f)(x) &= \sum_{j=-n}^{n} f(\phi_{\mathrm{DE}}(jh)) S(j,h)(\varphi_{\mathrm{DE}}(x)).
\end{align}
When $f$ is analytic and bounded on $\phi_{\mathrm{DE}}(\mathcal{S}_{\rho})$ for some $\rho\in(0,\pi/2)$, and $|f(x)|\leq \mathcal{K}|(x-a)^{\alpha}(b-x)^{\alpha}|$ for $x\in I$, then from \cite[Theorem~4.1]{Tanaka2009} we know that
\begin{align}\label{eq:DESincBound}
\|f - \mathcal{T}_{2n}^{\mathrm{DE}}f \|_{L^{\infty}(I)} \leq \mathcal{K} \exp\left(-\frac{\pi\rho n}{\log(2\rho n/\alpha)}\right),
\end{align}
where $h=\log(2\rho n/\alpha)/n$. Clearly, both DE-Sinc and DE-Hermite approximations converge at almost-exponential rate. 

Now we compare $\mathcal{T}_{2n}^{\mathrm{DE}}f$ and $\Pi_{2n,\lambda}^{\mathrm{DE}}f$ and both approximations have $2n+1$ terms. In the top row of Figure \ref{fig:Exam6} we plot the maximum errors of $\mathcal{T}_{2n}^{\mathrm{DE}}f$ and $\Pi_{2n,\lambda}^{\mathrm{DE}}f$ as a function of $n$ for $f(x)=x^{\alpha}(1-x)^{\beta}g(x)$ with $(\alpha,\beta)=(\pi/10,\pi/10)$ and $(\alpha,\beta)=(1/2,1/2)$ and $g(x)=1/(x+2)$. The step size in $\mathcal{T}_{2n}^{\mathrm{DE}}f$ is chosen as $h=\log(\pi n/\alpha)/n$ and the scaling factor in $\Pi_{2n,\lambda}^{\mathrm{DE}}f$ is chosen as $\lambda=\kappa\sqrt{2n}/\log(2n)$ with $\kappa=1.2$. We see that the convergence rate of $\Pi_{2n,\lambda}^{\mathrm{DE}}f$ is much faster than that of $\mathcal{T}_{2n}^{\mathrm{DE}}f$. In the bottom row of Figure \ref{fig:Exam6} we plot the maximum errors of $\mathcal{T}_{2n}^{\mathrm{DE}}f$ and $\Pi_{2n,\lambda}^{\mathrm{DE}}f$ as a function of $n$ for $f(x)=x^{\alpha}\log^{\beta}(x)$ with $(\alpha,\beta)=(1,1)$ and $(\alpha,\beta)=(2,2)$. The step size in $\mathcal{T}_{2n}^{\mathrm{DE}}f$ is still chosen as $h=\log(\pi n/\alpha)/n$ and the scaling factor in $\Pi_{2n,\lambda}^{\mathrm{DE}}f$ is chosen as $\lambda=\kappa\sqrt{2n}/\log(2n)$ with $\kappa=1.5$. We also see that the convergence rate of $\Pi_{2n,\lambda}^{\mathrm{DE}}f$ is much faster than that of $\mathcal{T}_{2n}^{\mathrm{DE}}f$.

\begin{figure}[htbp]
\centering
\includegraphics[width=0.49\textwidth,height=0.42\textwidth]{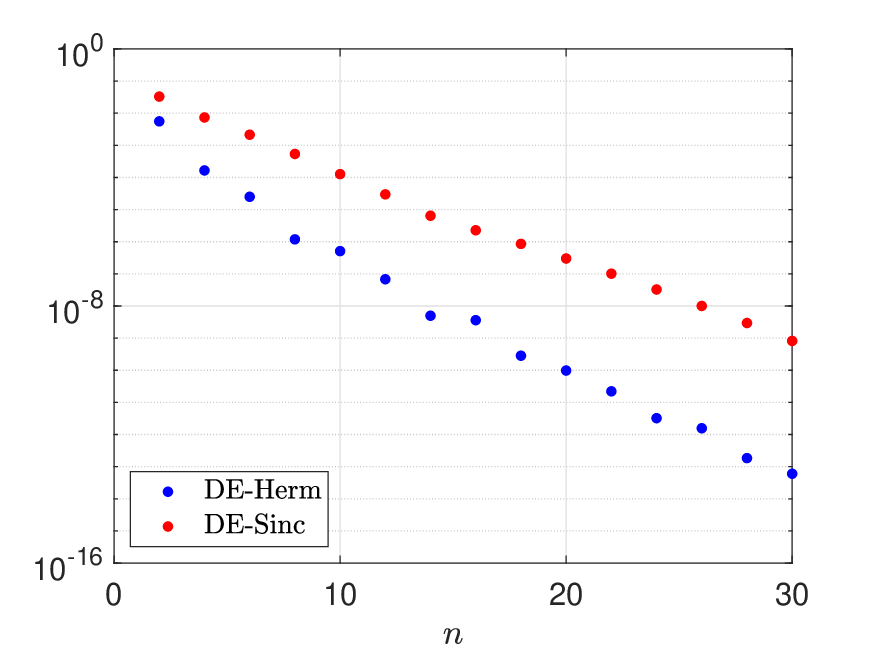}
\includegraphics[width=0.49\textwidth,height=0.42\textwidth]{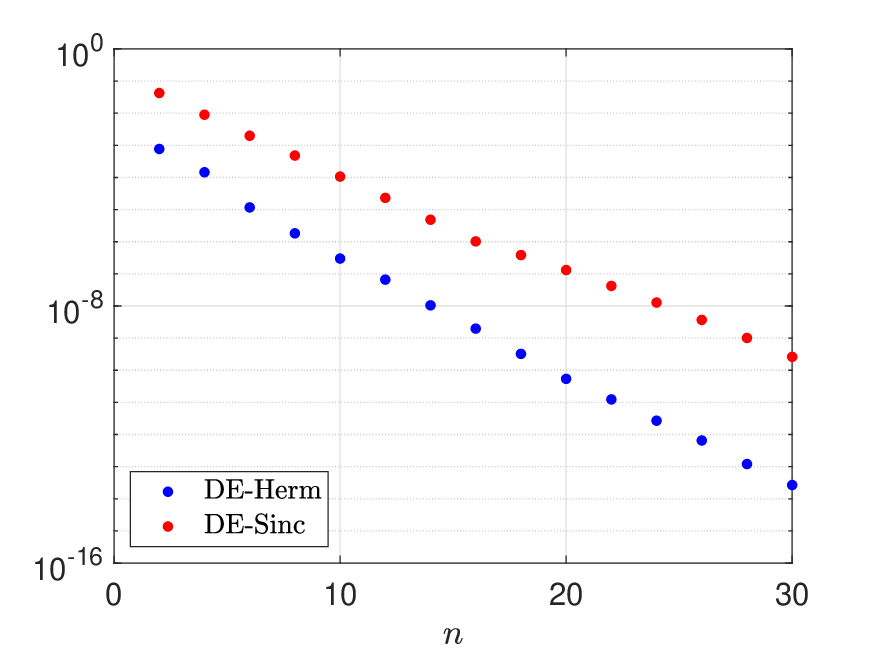}\\
\includegraphics[width=0.49\textwidth,height=0.42\textwidth]{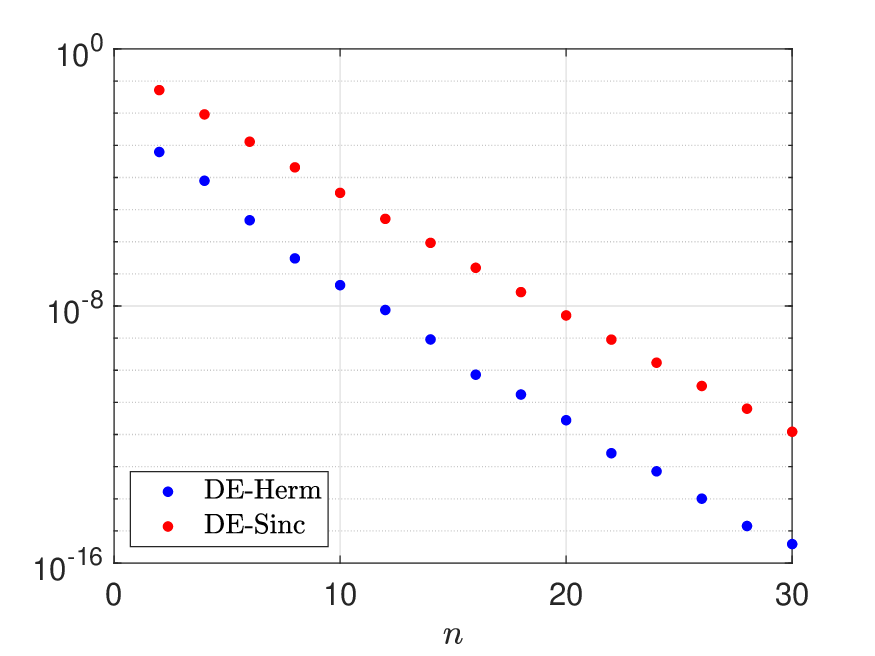}
\includegraphics[width=0.49\textwidth,height=0.42\textwidth]{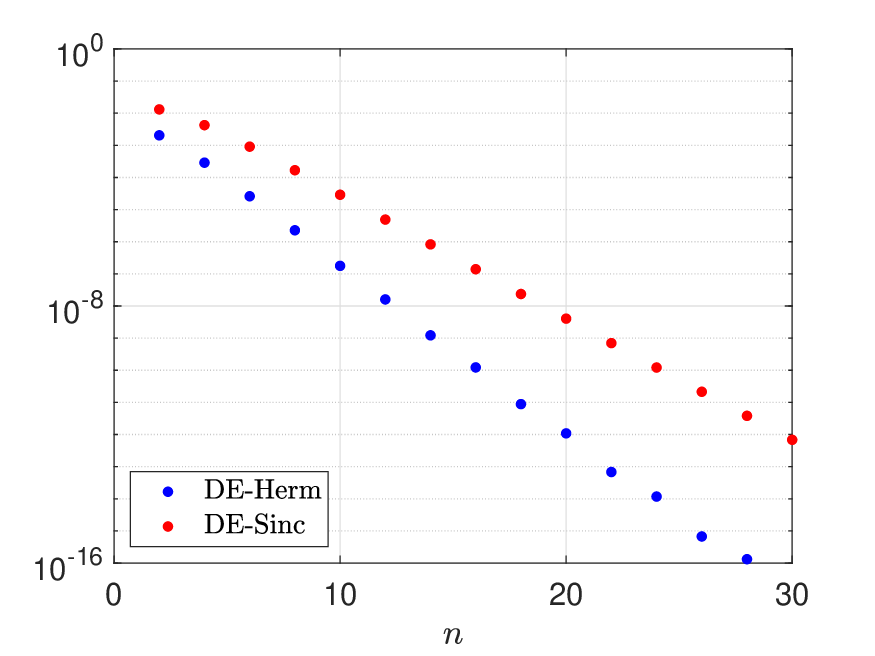}
\caption{Top: Maximum errors of $\Pi_{2n,\lambda}^{\mathrm{DE}}f$ (DE-Herm) and $\mathcal{T}_{2n}^{\mathrm{DE}}f$ (DE-Sinc) as a function of $n$ for $f(x)=x^{\alpha}(1-x)^{\beta}/(x+2)$ with $(\alpha,\beta)=(\pi/10,\pi/10)$ (left) and $(\alpha,\beta)=(1/2,3/4)$ (right). The scaling factor of $\Pi_{2n,\lambda}^{\mathrm{DE}}f$ is $\lambda=\kappa\sqrt{2n}/\log(2n)$ with $\kappa=1.2$. Bottom: Maximum errors of $\Pi_{2n,\lambda}^{\mathrm{DE}}f$ (DE-Herm) and $\mathcal{T}_{2n}^{\mathrm{DE}}f$ (DE-Sinc) as a function of $n$ for $f(x)=x^{\alpha}\log^{\beta}(x)$ with $(\alpha,\beta)=(1,1)$ (left) and $(\alpha,\beta)=(2,2)$ (right). The scaling factor of $\Pi_{2n,\lambda}^{\mathrm{DE}}f$ is $\lambda=\kappa\sqrt{2n}/\log(2n)$ with $\kappa=1.5$.}\label{fig:Exam6}
\end{figure}

%
%
%

\section{Hermite spectral approximation by interpolation}\label{sec:Interp}
In this section we develop an interpolation analog method of those projection approximations in section \ref{sec:ScaledHerm}. Let $\phi$ denote one of the SE, DE and EF transforms and let $\{x_j\}_{j=0}^{n}$ be a set of points in $I$ defined by $x_j=\phi(s_j/\lambda)$ and $\{s_j\}_{j=0}^{n}$ are the zeros of $\psi_{n+1}(x)$. The Hermite spectral interpolation is achieved by finding $I_{n,\lambda}f\in\mathbb{M}_{n}^{\lambda}$ such that
\begin{equation}\label{eq:ScaInt}
(I_{n,\lambda}f)(x_j) = f(x_j), \quad j=0,\ldots,n.
\end{equation}
Note that $(I_{n,\lambda}f)(\phi(s))$ is also the function which belongs to the space $\mathbb{H}_n^{\lambda}$ and interpolates $(f\circ\phi)(s)$ at the points $\{\varphi(x_j)\}_{j=0}^{n}$, and thus $e^{(\lambda s)^2/2}(I_{n,\lambda}f)(\phi(s))$ is the polynomial of degree $n$ which interpolates $e^{(\lambda s)^2/2}(f\circ\phi)(s)$ at the points $\{\varphi(x_j)\}_{j=0}^{n}$. By Lagrange interpolation formula, we immediately obtain
\begin{align}\label{eq:ScaIntLag}
(I_{n,\lambda}f)(x) = \sum_{j=0}^{n} f(x_j) L_j(x), \quad   L_j(x) = \frac{\psi_{n+1}(\lambda \varphi(x))}{\lambda (\varphi(x) - \varphi(x_j)) \psi'_{n+1}(s_j) }.
\end{align}
Below we show a connection between the error bounds of $I_{n,\lambda}f$ and $\Pi_{n,\lambda}f$ in $L^{\infty}$-norm.
\begin{lemma}\label{lem:ScaInt}
Let $I_{n,\lambda}f$ and $\Pi_{n,\lambda}f$ be defined in \eqref{eq:ScaInt} and \eqref{def:ScaledHermProj}, respectively. Then, it holds
\begin{align}
\|f - I_{n,\lambda}f\|_{L^{\infty}(I)} &\leq (1 + \Lambda_n^{\mathrm{Int}}) \|f - \Pi_{n,\lambda}f\|_{L^{\infty}(I)},
\end{align}
where $\Lambda_n^{\mathrm{Int}}$ is the Lebesgue constant of the interpolation operator $I_{n,\lambda}$, i.e.,
\[
\Lambda_n^{\mathrm{Int}} = \sup_{x\in I} \sum_{j=0}^{n} |L_j(x)|.
\]
\end{lemma}
\begin{proof}
By \eqref{eq:ScaIntLag} we know that  $\|I_{n,\lambda}f\|_{L^{\infty}(I)}\leq\Lambda_n^{\mathrm{Int}}\|f\|_{L^{\infty}(I)}$. On the other hand, note that $I_{n,\lambda}f\equiv f$ for $f\in\mathbb{Q}_{n,\lambda}$, we have
\begin{align}
\|f - I_{n,\lambda}f \|_{L^{\infty}(I)} &\leq \|f - \Pi_{n,\lambda}f \|_{L^{\infty}(I)} +
\|\Pi_{n,\lambda}f - I_{n,\lambda}f \|_{L^{\infty}(I)} \nonumber \\
&= \|f - \Pi_{n,\lambda}f \|_{L^{\infty}(I)} + \|I_{n,\lambda}(\Pi_{n,\lambda}f - f) \|_{L^{\infty}(I)} \nonumber \\
&\leq (1 + \Lambda_n^{\mathrm{Int}}) \|f - \Pi_{n,\lambda}f \|_{L^{\infty}(I)}. \nonumber
\end{align}
This ends the proof.
\end{proof}

As a direct consequence, we immediately derive the following results.
\begin{theorem}
Let $I_{n,\lambda}^{\mathrm{SE}}f$, $I_{n,\lambda}^{\mathrm{DE}}f$ and $I_{n,\lambda}^{\mathrm{EF}}f$ denote, respectively, the interpolation \eqref{eq:ScaInt} with $\phi=\phi_{\mathrm{SE}}$, $\phi=\phi_{\mathrm{DE}}$ and $\phi=\phi_{\mathrm{EF}}$. Then, the following results hold.
\begin{itemize}
\item[\rm(i)] If $f$ is analytic in $\phi_{\mathrm{SE}}(\mathcal{S}_{\rho})$ for some $\rho\in(0,\pi)$ and $|f(x)|\leq \mathcal{K} |(x-a)^{\alpha}(b-x)^{\beta}|$ for some $\alpha,\beta>0$ and $x\in I$. When choosing $\lambda=\sqrt{\tau/\rho}$ with $\tau=\min\{\alpha,\beta\}$, then
\begin{equation}\label{eq:ScaIntSEConv}
\|f - I_{n,\lambda}^{\mathrm{SE}}f\|_{L^{\infty}(I)} \leq \mathcal{K} n^{2/3} \exp(-\nu\sqrt{n}),
\end{equation}
where $\nu=\sqrt{2\tau\rho}$.

\item[\rm(ii)] If $f$ is analytic on $\phi_{\mathrm{DE}}(\mathcal{S}_{\rho})$ for some $\rho\in(0,\pi/2)$ and $|f(x)| \leq \mathcal{K} |(x-a)^{\alpha}(b-x)^{\beta}|$ for some $\alpha,\beta>0$ and $x\in{I}$. When choosing $\lambda=\kappa\sqrt{n}/\log n$ with $\kappa>0$, then
\begin{equation}\label{eq:ScaIntDEConv}
\|f - \Pi_{n,\lambda}^{\mathrm{DE}}f\|_{L^{\infty}(I)} \leq \mathcal{K} n^{1/6} \exp\left(-\nu n/\log n\right),
\end{equation}
for some $\nu>0$.

\item[\rm(iii)] If $f$ is analytic on $\phi_{\mathrm{EF}}(\mathcal{S}_{\rho})$ for some $\rho\in(0,\rho^{*})$ and $|f(x)|\leq \mathcal{K} |(x-a)^{\alpha}(b-x)^{\beta}|$ for some $\alpha,\beta>0$ and $x\in{I}$. When choosing $\lambda=\kappa n^{1/6}$ with $\kappa>0$, then
\begin{equation}\label{eq:ScaIntEFConv}
\|f - \Pi_{n,\lambda}^{\mathrm{EF}}f\|_{L^{\infty}(I)} \leq \mathcal{K} n^{1/6} \exp(-\nu n^{2/3}),
\end{equation}
for some $\nu>0$.
\end{itemize}
\end{theorem}
\begin{proof}
Let $x=\phi(s)$ with $s\in(-\infty,\infty)$, then
\[
L_j(x) = \frac{\psi_{n+1}(\lambda \varphi(x))}{\lambda (\varphi(x) - \varphi(x_j)) \psi'_{n+1}(s_j) } = \frac{\psi_{n+1}(\lambda s)}{(\lambda s - s_j) \psi'_{n+1}(s_j) }
\]
and thus
\begin{align}
\Lambda_n^{\mathrm{Int}} = \sup_{s\in \mathbb{R}} \sum_{j=0}^{n} \left| \frac{\psi_{n+1}(\lambda s)}{(\lambda s - s_j) \psi'_{n+1}(s_j) } \right| = \sup_{s\in \mathbb{R}} \left( w(\lambda s) \sum_{j=0}^{n} |\ell_j(\lambda s)| w^{-1}(s_j) \right), \nonumber
\end{align}
where $w(s)=e^{-s^2/2}$ and $\ell_j(s)$ are fundamental polynomials associated with $\{s_j\}_{j=0}^{n}$.
By \cite[Theorem~11.8]{Lubinsky2007} we know that $\Lambda_n^{\mathrm{Int}}=O(n^{1/6})$ as $n\rightarrow\infty$. The desired result \eqref{eq:ScaIntSEConv} follows immediately by combining Lemma \ref{lem:ScaInt} with Theorem \ref{thm:ScaledSEConv} and the desired results \eqref{eq:ScaIntDEConv} and \eqref{eq:ScaIntEFConv} follow immediately by combining Lemma \ref{lem:ScaInt} with Theorem \ref{thm:ScaledDE}.
\end{proof}

\begin{figure}[htbp]
\centering
\includegraphics[width=0.49\textwidth,height=0.42\textwidth]{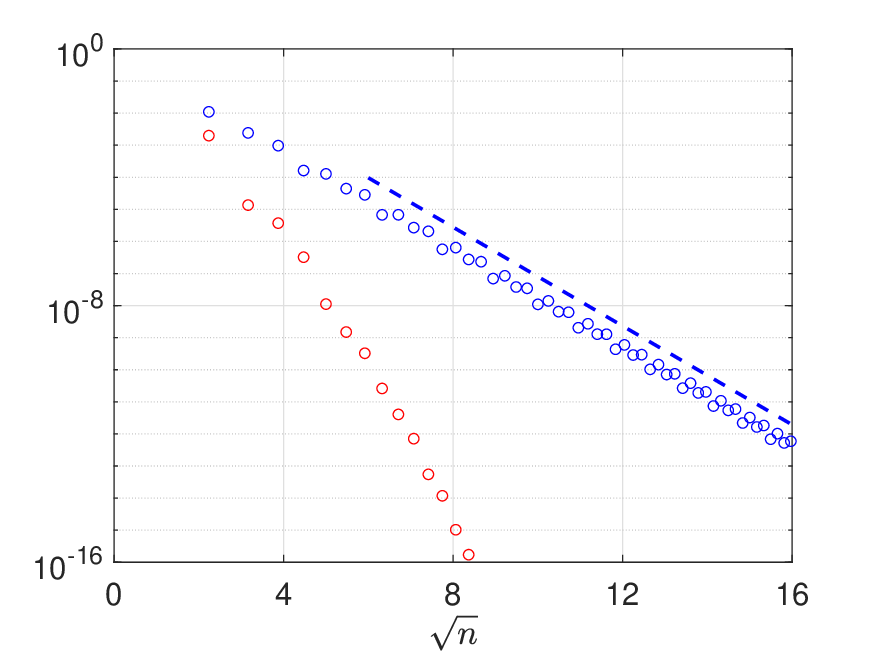}
\includegraphics[width=0.49\textwidth,height=0.42\textwidth]{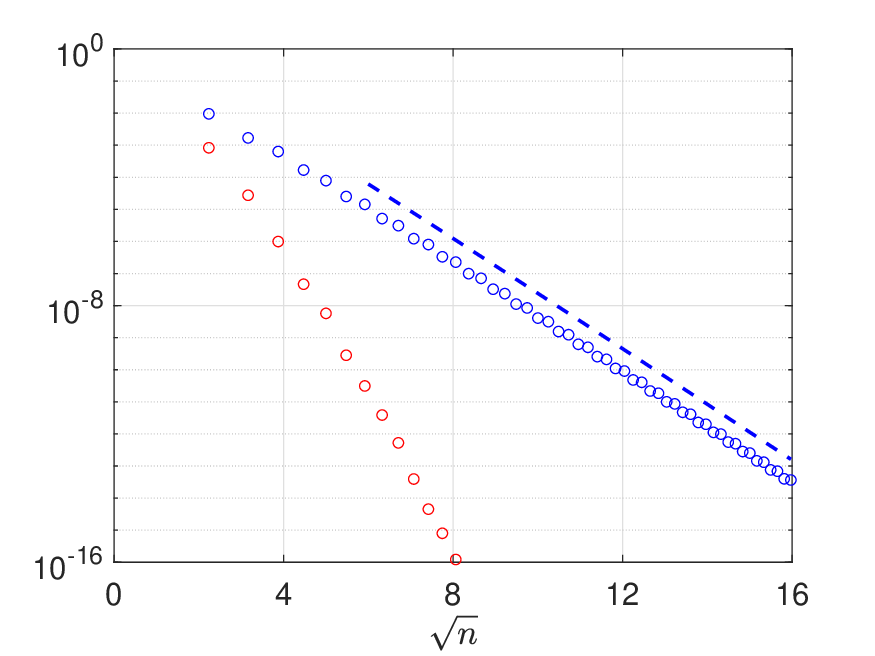}
\caption{Maximum errors of $I_{n,\lambda}^{\mathrm{SE}}f$ (blue) and $I_{n,\lambda}^{\mathrm{DE}}f$ (red) as a function of $\sqrt{n}$ for $f(x)=x^{\alpha}(1-x)^{\beta}$ with $(\alpha,\beta)=(1/2,1/2)$ (left) and $(\alpha,\beta)=(2/3,1)$ (right). The scaling factor of $I_{n,\lambda}^{\mathrm{SE}}f$ is $\lambda=\sqrt{\tau/\pi}$ with $\tau=\min\{\alpha,\beta\}$ and the scaling factor of $I_{n,\lambda}^{\mathrm{DE}}f$ is $\lambda=1.4\sqrt{n}/\log n$. The dashed lines show the rate $O(n^{\kappa}\exp(-\nu\sqrt{n}))$ with $\nu=\sqrt{2\tau\pi}$ and $\kappa=0$ (left) and $\kappa=1/3$ (right).}\label{fig:Exam8}
\end{figure}

Below we illustrate an example to show the performance of the SE- and DE-Hermite interpolation methods. In Figure \ref{fig:Exam8} we plot the maximum errors of $I_{n,\lambda}^{\mathrm{SE}}f$ and $I_{n,\lambda}^{\mathrm{DE}}f$ as a function of $\sqrt{n}$ for the function $f(x)=x^{\alpha}(1-x)^{\beta}$ with $(\alpha,\beta)=(1/2,1/2)$ and $(\alpha,\beta)=(2/3,1)$. The scaling factor of $I_{n,\lambda}^{\mathrm{SE}}f$ is chosen as $\lambda=\sqrt{\tau/\pi}$ with $\tau=\min\{\alpha,\beta\}$ and the scaling factor of $I_{n,\lambda}^{\mathrm{DE}}f$ is chosen as $\lambda=1.4\sqrt{n}/\log n$. We see from Figure \ref{fig:Exam8} that $I_{n,\lambda}^{\mathrm{SE}}f$ converges, up to some small algebraic factors, at the rate $O(\exp(-\nu\sqrt{n}))$ with $\nu=\sqrt{2\tau\pi}$ and $I_{n,\lambda}^{\mathrm{DE}}f$ converges at some much faster rates. Finally, comparing Figure \ref{fig:Exam8} with Figure \ref{fig:Exam3}, we see that the accuracy of $I_{n,\lambda}^{\mathrm{SE}}f$ is very close to that of $\Pi_{n,\lambda}^{\mathrm{SE}}f$. Similarly, the accuracy of $I_{n,\lambda}^{\mathrm{DE}}f$ is also very close to that of $\Pi_{n,\lambda}^{\mathrm{DE}}f$ (not displayed here).

Before closing this section, we point out an interesting observation on the distribution of the interpolation points $\{x_j\}_{j=0}^{n}$. In Figure \ref{fig:Point} we plot the points $\{x_j\}_{j=0}^{n}$ as a function of $\sqrt{j}$ for $\phi=\phi_{\mathrm{SE}}$ and $\phi=\phi_{\mathrm{DE}}$ and $n=36$ and $I=(0,1)$. We choose $\lambda=1/\sqrt{2\pi}$ for $\phi=\phi_{\mathrm{SE}}$ and $\lambda=1.5\sqrt{n}/\log n$ for $\phi=\phi_{\mathrm{DE}}$. For those points near the left endpoint $x=0$, we see that the curve of $\phi=\phi_{\mathrm{SE}}$ is close to a parabola and the curve of $\phi=\phi_{\mathrm{DE}}$ is close to a straight line on a semi-log scale, and thus they exhibit uniform exponential clustering for $\phi=\phi_{\mathrm{SE}}$ and tapered exponential clustering for $\phi=\phi_{\mathrm{DE}}$. This observation is similar to the observation in \cite[Figure~13]{Trefethen2021} on the distances of the nodes of SE-Sinc and DE-Sinc quadrature from the left endpoint $x=-1$.

\begin{figure}[htbp]
\centering
\includegraphics[width=0.55\textwidth,height=0.45\textwidth]{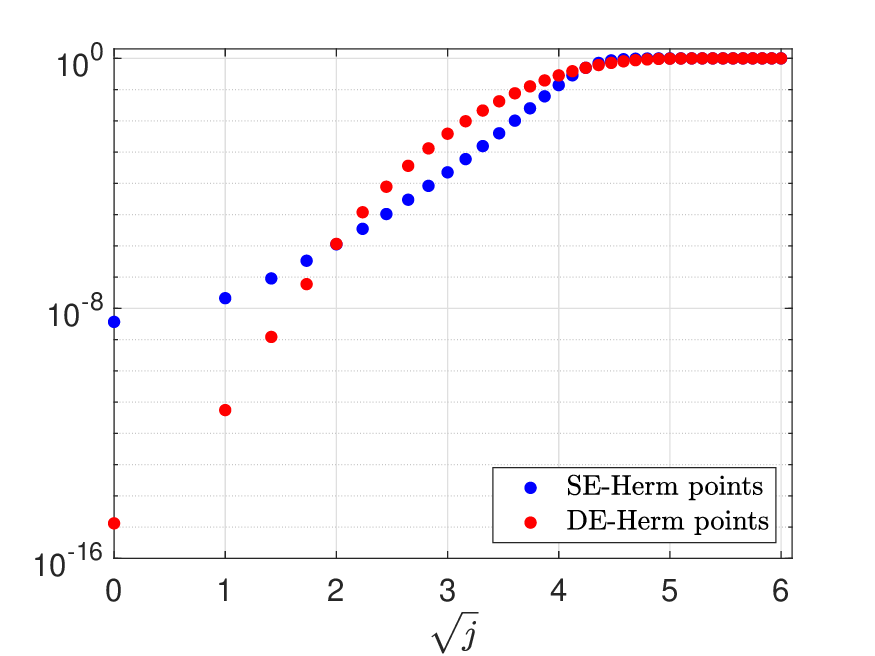}
\caption{The interpolation points $\{x_j\}_{j=0}^{n}$ as a function of $\sqrt{j}$ for $\phi=\phi_{\mathrm{SE}}$ (blue) and $\phi=\phi_{\mathrm{DE}}$ (red) and $n=36$. The scaling factor is chosen as $\lambda=1/\sqrt{2\pi}$ for $\phi=\phi_{\mathrm{SE}}$ and $\lambda=1.5\sqrt{n}/\log n$ for $\phi=\phi_{\mathrm{DE}}$.}\label{fig:Point}
\end{figure}

\section{Extensions}\label{sec:Exten}
In this section we present two extensions of SE- and DE-Hermite approximations, including SE- and DE-Hermite quadrature and the rootfinding algorithm for computing the roots and extrema of functions with endpoint singularities.

\subsection{SE- and DE-Hermite quadrature}
Consider the computation of integral
\begin{equation}\label{eq:Int}
I(f) = \int_{a}^{b} f(x) dx,
\end{equation}
where $f(x)$ is either smooth or has one or two endpoint singularities. SE- and DE-Sinc quadrature for the above integral has been extensively studied and the idea is to apply the SE and DE transforms to the above integral and then approximate the resulting integral by trapezoidal rule, i.e.,
\begin{equation}
I(f) = \int_{-\infty}^{\infty} f(\phi(s)) \phi'(s) ds \approx h \sum_{k=-M}^{N} f(\phi(kh)) \phi'(kh),
\end{equation}
where $h>0$ is the step size. The sum on the right-hand side is called SE-Sinc quadrature when $\phi=\phi_{\mathrm{SE}}$ and DE-Sinc quadrature when $\phi=\phi_{\mathrm{DE}}$. We refer to \cite{Okayama2013,Stenger1993,Takahasi1973,Takahasi1974} for more details.

Here we develop SE- and DE-Hermite quadrature formulas for the integral \eqref{eq:Int}. Let $g(x)=f(x)/\varphi'(x)$, using the interpolation approximation \eqref{eq:ScaInt} to $g(x)$, we define an interpolatory quadrature as
\begin{equation}\label{eq:HermQuad}
I(f) = \int_{a}^{b} g(x) \varphi'(x) dx
\approx \int_{a}^{b} (I_{n,\lambda}g)(x) \varphi'(x) dx  = \sum_{j=0}^{n} w_j f(x_j) := Q_{n,\lambda}(f).
\end{equation}
By \eqref{eq:ScaIntLag} and using the change of variable $s=\varphi(x)$, the weights are given by
\begin{align}
w_j &= \frac{1}{\varphi'(x_j)} \int_{a}^{b} L_j(x) \varphi'(x) dx = \frac{1}{\varphi'(x_j) \psi'_{n+1}(s_j)} \int_{-\infty}^{\infty} \frac{\psi_{n+1}(\lambda s)}{\lambda s - s_j} ds  \nonumber \\
&= \frac{1}{\varphi'(x_j) \psi'_{n+1}(s_j) \psi_n(s_j)} \sqrt{\frac{2}{n+1}} \sum_{k=0}^{n} \psi_k(s_j) \int_{-\infty}^{\infty} \psi_k(\lambda s) ds,
\end{align}
and the integrals in the last step are known in closed form. Note that $Q_{n,\lambda}(f)$ is constructed by approximating $g$ with $I_{n,\lambda}g$, one would expect that the convergence rate of $Q_{n,\lambda}(f)$ should be the same as that of $I_{n,\lambda}g$.

Below we denote $Q_{n,\lambda}(f)$ by $Q_{n,\lambda}^{\mathrm{SE}}(f)$ when $\phi=\phi_{\mathrm{SE}}$ and by $Q_{n,\lambda}^{\mathrm{DE}}(f)$ when $\phi=\phi_{\mathrm{DE}}$. In Figure \ref{fig:Exam9} we plot their errors for $f(x)=x^{\alpha}\log^{\beta}(x) e^{x}$ with $(\alpha,\beta)=(-1/5,2)$ and $(\alpha,\beta)=(-1/3,1)$ and the interval is $I=(0,1)$. Note that $g(x)=O((x-a)^{\alpha+1})$ as $x\rightarrow a^+$ and $g(x)=O((b-x)^{\beta+1})$ as $x\rightarrow b^-$ for both SE and DE transforms, and thus we choose $\lambda=\sqrt{\tau/\pi}$ with $\tau=\min\{\alpha,\beta\}+1$ in $Q_{n,\lambda}^{\mathrm{SE}}(f)$. As for $Q_{n,\lambda}^{\mathrm{DE}}(f)$, we choose $\lambda=1.5\sqrt{n}/\log n$. As expected, we see that $Q_{n,\lambda}^{\mathrm{SE}}(f)$ converges, up to some small algebraic factors, at the rate $O(\exp(-\nu\sqrt{n}))$ with $\nu=\sqrt{2\tau\pi}$ and $Q_{n,\lambda}^{\mathrm{DE}}(f)$ converges at a much faster rate.

\begin{figure}[htbp]
\centering
\includegraphics[width=0.49\textwidth,height=0.42\textwidth]{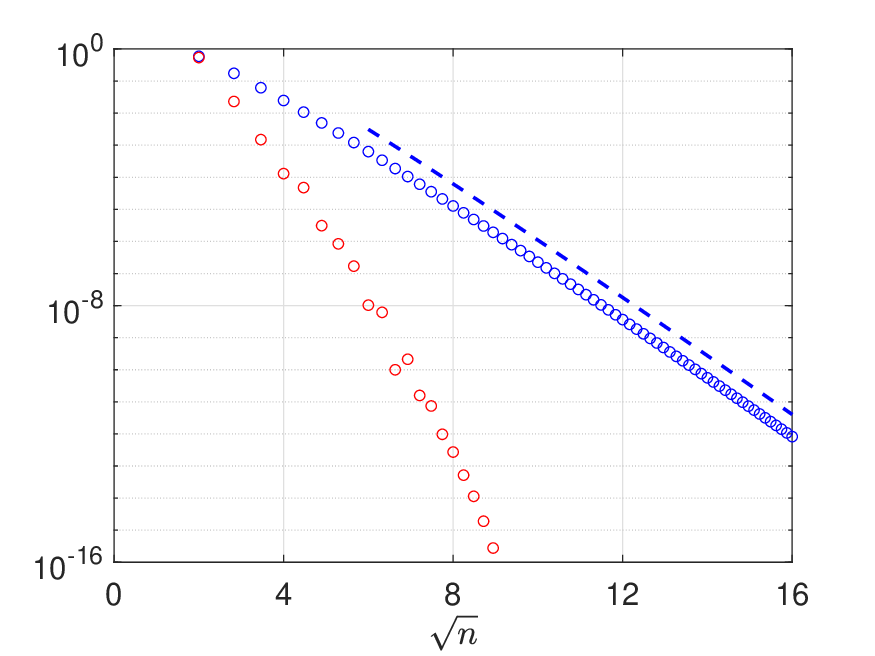}
\includegraphics[width=0.49\textwidth,height=0.42\textwidth]{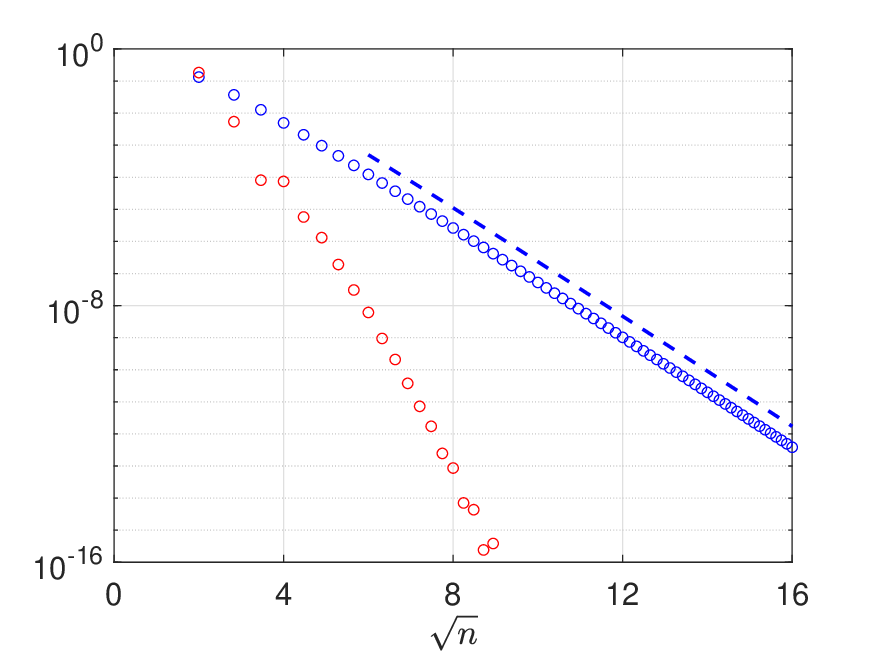}
\caption{Absolute errors of $Q_{n,\lambda}^{\mathrm{SE}}(f)$ (blue) and $Q_{n,\lambda}^{\mathrm{DE}}(f)$ (red) as a function of $\sqrt{n}$ for $f(x)=x^{\alpha}\log^{\beta}(x) e^{x}$ and $(\alpha,\beta)=(-1/5,2)$ (left) and $(\alpha,\beta)=(-1/3,1)$ (right). The scaling factor of $Q_{n,\lambda}^{\mathrm{SE}}(f)$ is $\lambda=\sqrt{\tau/\pi}$ with $\tau=\min\{\alpha,\beta\}+1$ and the scaling factor of $Q_{n,\lambda}^{\mathrm{DE}}(f)$ is $\lambda=1.5\sqrt{n}/\log n$. The dashed lines show the rate $O(n^{\kappa}\exp(-\nu\sqrt{n}))$ with $\nu=\sqrt{2\tau\pi}$ and $\kappa=1$ (left) and $\kappa=1/2$ (right).}\label{fig:Exam9}
\end{figure}

\subsection{SE- and DE-Hermite rootfinding algorithm}
Chebyshev rootfinding algorithm is a stable and efficient numerical method for computing the roots and extrema \cite[Chapter~18]{Trefethen2019}. The performance of this algorithm depends on the accuracy of Chebyshev spectral approximation and differentiation and thus the convergence rate will be slow when the function has singularities (see, e.g., \cite{Boyd2013,Wang2026b}).

Below we consider the rootfinding algorithm using SE- and DE-Hermite approximations. First, we approximate $f(x)$ by Hermite approximation of the form
\begin{equation}
f(x) \approx p(x) = \sum_{k=0}^{n} a_k \Psi_{k,\lambda}(x),
\end{equation}
and $p$ is derived either by projection \eqref{def:ScaledHermProj} or interpolation \eqref{eq:ScaIntLag}. By the recurrence relation of $\{\psi_k\}$, it is easy to show that if $p(x)=0$, then $\lambda\varphi(x)$ is the eigenvalue of the following companion matrix
\begin{align}
A = \left(
             \begin{array}{cccc}
                 0 & \sqrt{\frac{1}{2}} &    &   \\
               \sqrt{\frac{1}{2}} & \ddots & \ddots &   \\
                 & \ddots & \ddots & \sqrt{\frac{n-1}{2}} \\
                 &  & \sqrt{\frac{n-1}{2}} & 0 \\
             \end{array}
           \right) - \sqrt{\frac{n}{2}} \frac{1}{a_n}
           \left(
             \begin{array}{cccc}
                 &   &   &   \\
                 &   &   &   \\
                 &   &   &   \\
               a_0 & a_1 & \cdots & a_{n-1} \\
             \end{array}
           \right),
\end{align}
and the entries not displayed are all zero. Thus, all roots of $p(x)$ in the interval $I$ can be evaluated from the eigenvalues of the above matrix. Similarly, by the derivative relation $\psi'_n(x) = \sqrt{n/2} \psi_{n-1}(x) - \sqrt{(n+1)/2} \psi_{n+1}(x)$, we have
\begin{equation}
p'(x) = \lambda\varphi'(x) \sum_{k=0}^{n+1} b_k \Psi_{k,\lambda}(x),
\end{equation}
where
\begin{align}
b_k &= \sqrt{\frac{k+1}{2}} a_{k+1} - \sqrt{\frac{k}{2}} a_{k-1} , \quad k=0,\ldots,n+1, \nonumber
\end{align}
and $a_{n+1}=a_{n+2}=0$. For both SE and DE transforms, i.e., $\varphi=\varphi_{\mathrm{SE}}$ and $\varphi=\varphi_{\mathrm{DE}}$, it is easily verified that $\varphi'(x)>0$ for $x\in I$, and thus the root of $p'(x)$ can also be evaluated by using the above algorithm in a similar way.

Below we show the performance of SE- and DE-Hermite rootfinding algorithm. Consider the function
\begin{equation}\label{eq:ExamZero}
f(x) = x^{1/2}(1-x)^{1/2}\left(x-\frac{1}{4}\right)\left(x-\frac{1}{2}\right)\left(x-\frac{3}{4}\right), \quad
\end{equation}
which has three roots and four extrema in $I=(0,1)$. In Figure \ref{fig:Exam10} we plot the maximum errors of the above algorithm for computing the three roots and four extrema of the function \eqref{eq:ExamZero}. We choose $\lambda=1/\sqrt{2\pi}$ for $\varphi=\varphi_{\mathrm{SE}}$ and $\lambda=1.6\sqrt{n}/\log n$ for $\varphi=\varphi_{\mathrm{DE}}$. As expected, the algorithm with $\varphi=\varphi_{\mathrm{SE}}$ converges, up to an algebraic factor, at the rate $O(\exp(-\nu\sqrt{n}))$ with $\nu=\sqrt{\pi}$ and the algorithm with $\varphi=\varphi_{\mathrm{DE}}$ converges at a much faster rate.

\begin{figure}[htbp]
\centering
\includegraphics[width=0.49\textwidth,height=0.42\textwidth]{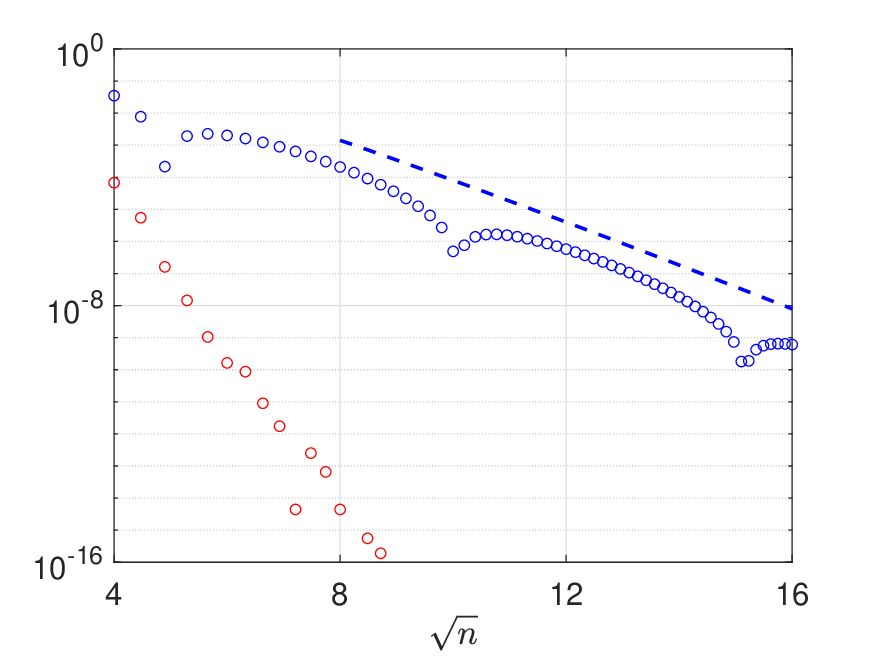}
\includegraphics[width=0.49\textwidth,height=0.42\textwidth]{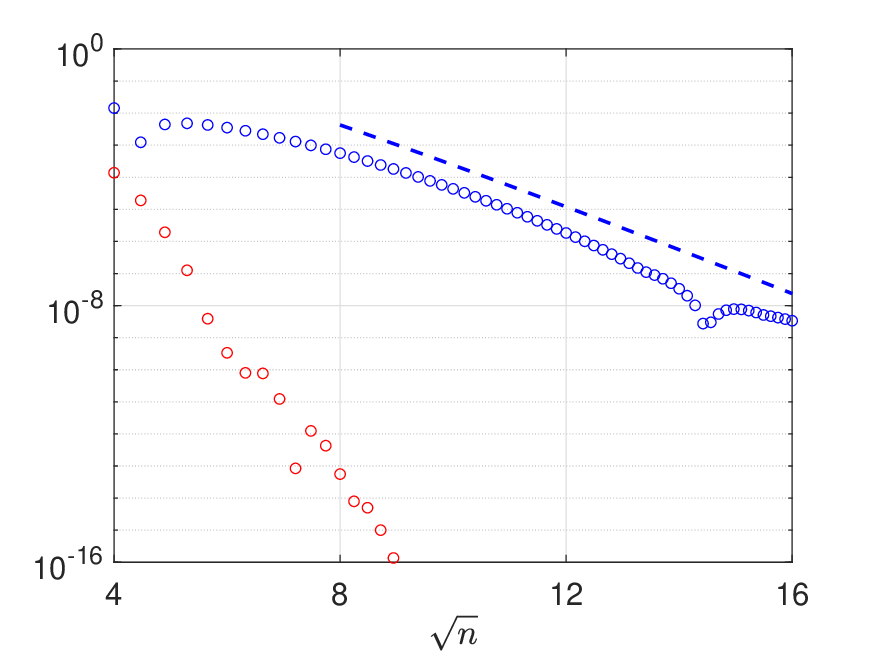}
\caption{Maximum errors of the SE-Hermite (blue) and DE-Hermite (red) rootfinding algorithm for the three roots (left) and the four extrema (right) as a function of $\sqrt{n}$. The scaling factor is $\lambda=1/\sqrt{2\pi}$ for $\varphi=\varphi_{\mathrm{SE}}$ and $\lambda=1.6\sqrt{n}/\log n$ for $\varphi=\varphi_{\mathrm{DE}}$. The dashed lines show the rate $O(n^{\kappa}\exp(-\nu\sqrt{n}))$ with $\nu=\sqrt{\pi}$ and $\kappa=3/2$. }\label{fig:Exam10}
\end{figure}

\section{Concluding remarks}\label{sec:conclusion}
In this work we have introduced Hermite spectral approximation for functions with endpoint singularities using exponential transforms and presented a comprehensive convergence analysis of such approximations without and with scalings. In the case with scalings, we have derived optimal scaling factors and showed that SE- and DE-Hermite approximations have comparable or superior accuracy performance to their sinc counterparts when using the same number of terms.

Finally, we list several issues for future research: 
\begin{itemize}
\item Quantifying constant factor in the scaling factors of DE- and EF-Hermite approximations.

\item Developing SE- and DE-Hermite spectral methods for Fredholm and Volterra integral equations with singular kernels.
\end{itemize}
We will address these issues in the future.



\appendix

\section{Proof of Theorem \ref{thm:ScaledHermInf}}\label{app:A}

\begin{proof}
We follow the steps as the proof of \cite[Theorem 2.4]{Hu2026}. Let $G(x) = e^{-x^2/2}/\sqrt{2\pi}$ and $h_M(x) = \mathbb{I}_{[-2M,2M]}$ and $T_M(x) = (h_M\ast G)(x)$. It is easily verified that $0\leq T_M(x)\leq1$ for $x\in(-\infty,\infty)$ and
\[
0\leq 1 - T_M(x) \leq \sqrt{\frac{2}{\pi}} \frac{e^{-M^2/2}}{M} , \quad x\in[-M,M].
\]
We first consider the error of $\Pi_{n,1}^{\mathrm{SH}}$ in $L^{\infty}$-norm. Let $u_{M}(x) = u(x)T_{M}(x)$, we have
\begin{align}\label{eq:AppendEqn1}
\|u - \Pi_{n,1}^{\mathrm{SH}} u\|_{L^{\infty}(\mathbb{R})} &\leq \|u - u_{M}\|_{L^{\infty}(\mathbb{R})} + \|u_{M} - \Pi_{n,1}^{\mathrm{SH}} u_{M}\|_{L^{\infty}(\mathbb{R})} + \|\Pi_{n,1}^{\mathrm{SH}} u_{M} - \Pi_{n,1}^{\mathrm{SH}} u\|_{L^{\infty}(\mathbb{R})}  \nonumber \\[1ex]
&\leq (1+ \Lambda_n) \|u - u_{M}\|_{L^{\infty}(\mathbb{R})} + \|u_{M} - \Pi_{n,1}^{\mathrm{SH}} u_{M}\|_{L^{\infty}(\mathbb{R})},
\end{align}
where $\Lambda_n$ is the Lebesgue constant of the operator $\Pi_{n,1}^{\mathrm{SH}}$ in $L^{\infty}$-norm. For the first term in the last inequality, using the properties of $T_M(x)$ we have
\begin{align}\label{eq:AppendEqn11}
\|u - u_M\|_{L^{\infty}(\mathbb{R})} &\leq \| u (1 - T_M) \cdot \mathbb{I}_{\{|x|<{M}\}} \|_{L^{\infty}(\mathbb{R})} + \| u (1 - T_{M}) \cdot \mathbb{I}_{\{|x|\geq M\}} \|_{L^{\infty}(\mathbb{R})} \nonumber \\
&\leq \sqrt{\frac{2}{\pi}} \frac{e^{-M^2/2}}{M} \|u\|_{L^{\infty}(\mathbb{R})} + \| u \cdot \mathbb{I}_{\{|x|\geq M\}} \|_{L^{\infty}(\mathbb{R})}.
\end{align}
Now we consider the second term. Let $\mathcal{B}_{N}$ denote the restriction operator in frequency domain of the form
\[
(\mathcal{B}_{N}u)(x) = \frac{1}{\sqrt{2\pi}} \int_{-N}^{N} \hat{u}(\xi) e^{\mathrm{i}\xi x} d\xi = \mathcal{F}^{-1}[\hat{u}(\xi) \cdot \mathbb{I}_{\{|\xi|\leq N\}}](x),
\]
and let $u_N(x)=(\mathcal{B}_{N}u)(x)T_M(x)$. Similar to \eqref{eq:AppendEqn1}, we have
\begin{align}\label{eq:AppendEqn2}
\|u_M - \Pi_{n,1}^{\mathrm{SH}} u_M\|_{L^{\infty}(\mathbb{R})} &\leq \|u_M - u_N\|_{L^{\infty}(\mathbb{R})} + \|u_N - \Pi_{n,1}^{\mathrm{SH}} u_N\|_{L^{\infty}(\mathbb{R})} \nonumber \\[1ex]
& + \|\Pi_{n,1}^{\mathrm{SH}} u_N - \Pi_{n,1}^{\mathrm{SH}} u_M \|_{L^{\infty}(\mathbb{R})} \nonumber \\[1ex]
&\leq (1 + \Lambda_n) \|u_M - u_N\|_{L^{\infty}(\mathbb{R})} + \|u_N - \Pi_{n,1}^{\mathrm{SH}} u_N\|_{L^{\infty}(\mathbb{R})}.
\end{align}
For the first term in the last inequality of \eqref{eq:AppendEqn2}, we have
\begin{align}
\|u_M - u_N\|_{L^{\infty}(\mathbb{R})} &= \|(u - \mathcal{B}_{N}u) \cdot T_M\|_{L^{\infty}(\mathbb{R})} \leq \|u - \mathcal{B}_{N}u \|_{L^{\infty}(\mathbb{R})}. \nonumber
\end{align}
Furthermore, by the Fourier inversion formula,
\begin{align}
(u - \mathcal{B}_{N}u)(x) &= \frac{1}{\sqrt{2\pi}} \int_{-\infty}^{\infty} \hat{u}(\xi) e^{\mathrm{i}\xi x} d\xi - \frac{1}{\sqrt{2\pi}} \int_{-N}^{N} \hat{u}(\xi) e^{\mathrm{i}\xi x} d\xi = \frac{1}{\sqrt{2\pi}} \int_{|\xi|\geq N} \hat{u}(\xi) e^{\mathrm{i}\xi x} d\xi,  \nonumber
\end{align}
it follows
\begin{align}\label{eq:AppendEqn3}
\|u - \mathcal{B}_{N}u \|_{L^{\infty}(\mathbb{R})} &\leq \frac{1}{\sqrt{2\pi}} \int_{|\xi|\geq N} \left| \hat{u}(\xi) \right| d\xi = \frac{1}{\sqrt{2\pi}} \|\hat{u}(\xi)\cdot \mathbb{I}_{\{|\xi|\geq N\}} \|_{L^1(\mathbb{R})} .
\end{align}
Next we consider the second term in the last inequality of \eqref{eq:AppendEqn2}. Note that
\begin{align}
u_N(x) = \frac{1}{2\pi} \left( \int_{-N}^{N} \hat{u}(\xi) e^{\mathrm{i}\xi x} d\xi \right) \left( \int_{-2M}^{2M} e^{-(x-y)^2/2} dy \right) = \frac{1}{2\pi} \int_{\Omega} \hat{u}(\xi) e^{\mathrm{i}\xi x -(x-y)^2/2} dy  d\xi, \nonumber
\end{align}
where $\Omega=\{(\xi,y):~|\xi|\leq N, ~|y|\leq 2M \}$, and by \cite[Equation~(2.27)]{Hu2026},
\begin{align}
u_N(x) &= \sum_{k=0}^{\infty} \left( \frac{1}{2\pi} \int_{\Omega} \hat{u}(\xi) c_k(\xi,y) dy d\xi \right) \psi_k(x),  \nonumber
\end{align}
where
\[
c_k(\xi,y) = \pi^{1/4} e^{-y^2/2-(\xi-\mathrm{i}y)^2/4 } \frac{(y+\mathrm{i}\xi)^k}{\sqrt{2^k k!}}.
\]
Therefore, recalling that $\|\psi_k\|_{L^{\infty}(\mathbb{R})}\leq1/\pi^{1/4}$ for any $k\geq0$,
\begin{align}\label{eq:AppendEqn4}
\|u_N - \Pi_{n,1}^{\mathrm{SH}} u_N\|_{L^{\infty}(\mathbb{R})} &\lesssim \sum_{k=n+1}^{\infty}  \left| \int_{\Omega} \hat{u}(\xi) c_k(\xi,y) dy  d\xi \right|  \nonumber \\
&\lesssim \|\hat{u}\|_{L^{\infty}(\mathbb{R})} \sum_{k=n+1}^{\infty} \int_{\Omega} \left| c_k(\xi,y) \right| dy d\xi .
\end{align}
Let $D=\{(\xi,y): \xi^2+y^2\leq R^2\}$ with $R=\sqrt{4M^2+N^2}=\sqrt{(4\sigma^2+\eta^2)n}$, it is easily seen that $\Omega\subseteq D$. For $k\geq n+1$ and $4\sigma^2+\eta^2\in(0,2)$, it is straightforward to check that $e^{-x^2/4} x^k \leq e^{-R^2/4} R^k $ for $x\in[0,R]$, we have for $(\xi,y)\in\Omega\subseteq D$ that
\begin{align}
\left| c_k(\xi,y) \right| = \frac{\pi^{1/4} (y^2+\xi^2)^{k/2} }{\sqrt{2^k k!}} e^{-(y^2+\xi^2)/4}  \leq \frac{\pi^{1/4} R^{k} }{\sqrt{2^k k!}} e^{-R^2/4},  \nonumber
\end{align}
and thus,
\begin{align}
\sum_{k=n+1}^{\infty} \int_{\Omega} \left| c_k(\xi,y) \right| dy d\xi
&\lesssim n e^{-R^2/4} \sum_{k=n+1}^{\infty} \frac{R^{k} }{\sqrt{2^k k!}}. \nonumber
\end{align}
By Stirling formula, we know that
\begin{align}
\sum_{k=n+1}^{\infty} \frac{R^{k}}{\sqrt{2^k k!}} \lesssim \frac{1}{n^{1/4}} \sum_{k=n+1}^{\infty} \left( \frac{e R^2}{2k} \right)^{k/2} &\lesssim \frac{1}{n^{1/4}} \left( \frac{e R^2}{2n} \right)^{n/2}. \nonumber
\end{align}
Hence,
\begin{align}
\sum_{k=N+1}^{\infty} \int_{\Omega} \left| c_k(\xi,y) \right| dy d\xi
&\lesssim n^{3/4} e^{-R^2/4} \left( \frac{e R^2}{2n} \right)^{n/2} .   \nonumber
\end{align}
Combing the above estimate with \eqref{eq:AppendEqn4}, \eqref{eq:AppendEqn3}, \eqref{eq:AppendEqn2}, \eqref{eq:AppendEqn11} and \eqref{eq:AppendEqn1} yields
\begin{align}
\|u - \Pi_{n,1}^{\mathrm{SH}} u\|_{L^{\infty}(\mathbb{R})}  &\lesssim (1 + \Lambda_n) \left( \| u \cdot \mathbb{I}_{\{|x|\geq M\}} \|_{L^{\infty}(\mathbb{R})} + \|\hat{u} \cdot \mathbb{I}_{\{|\xi|\geq N\}} \|_{L^1(\mathbb{R})} \right) \nonumber \\
& + (1 + \Lambda_n) \frac{e^{-M^2}}{M} \|u\|_{L^{\infty}(\mathbb{R})} + n^{3/4} e^{-R^2/4} \left( \frac{e R^2}{2n} \right)^{n/2} \|\hat{u}\|_{L^{\infty}(\mathbb{R})}. \nonumber
\end{align}
Finally, for the approximation $\Pi_{n,\lambda}^{\mathrm{SH}} u$, note that
\begin{align}
\|u - \Pi_{n,\lambda}^{\mathrm{SH}} u\|_{L^{\infty}(\mathbb{R})} = \|\mathcal{A}_{\lambda}u - \Pi_{n,1}^{\mathrm{SH}}(\mathcal{A}_{\lambda}u)\|_{L^{\infty}(\mathbb{R})},  \nonumber
\end{align}
where $\mathcal{A}_{\lambda}$ is the scaling operator defined in the proof of Theorem \ref{thm:ScaledSEConv}.
The desired result follows immediately. This ends the proof.
\end{proof}

\end{document}